\documentclass[11pt]{amsart}
\usepackage[leqno]{amsmath}
\usepackage{amssymb}
\usepackage{enumerate}
\usepackage{xcolor}
\usepackage{subfigure}
\usepackage{graphicx}
\usepackage{hyperref}
\usepackage{microtype}
\usepackage[all]{xy}
\numberwithin{equation}{section}

\newtheorem{theorem}{Theorem}[section]
\newtheorem{remark}[theorem]{Remark}
\newtheorem{lemma}[theorem]{Lemma}
\newtheorem{proposition}[theorem]{Proposition}

\newtheorem{definition}[theorem]{Definition}

\newcommand{\B}{\mathbf{B}}
\newcommand{\C}{\mathbf{C}}
\newcommand{\D}{\mathbf{D}}
\newcommand{\E}{\mathbf{E}}

\newcommand{\h}{\mathbf{H}}
\newcommand{\N}{\mathbf{N}}
\newcommand{\Z}{\mathbf{Z}}
\newcommand{\p}{\mathbf{P}}

\newcommand{\R}{\mathbf{R}}

\newcommand{\HH}{\mathbf{H}}

\newcommand{\varpsi}{\psi}

\newcommand{\CF}{\mathcal {F}}

\newcommand{\CK}{\mathcal {K}}
\newcommand{\CL}{\mathcal {L}}

\newcommand{\CN}{\mathcal {N}}

\newcommand{\CT}{\mathcal {T}}

\newcommand{\eps}{\varepsilon}
\newcommand{\SLE}{{\rm SLE}}
\newcommand{\CLE}{{\rm CLE}}

\newcommand{\wt}{\widetilde}
\newcommand{\ol}{\overline}
\newcommand{\ul}{\underline}

\newcommand{\BCLE}{\mathrm{BCLE}}
\newcommand{\ccwBCLE}{\BCLE^{\boldsymbol {\circlearrowleft}}}
\newcommand{\cwBCLE}{\BCLE^{\boldsymbol {\circlearrowright}}}
\newcommand{\wh}{\widehat}

\begin{document}

\title{Non-simple $\SLE$ curves are not determined by their range}

\author{Jason Miller}
\address{Statslab, Center for Mathematical Sciences, University of Cambridge, Wilberforce Road, Cambridge CB3 0WB, UK}
\email {jpmiller@statslab.cam.ac.uk}

\author{Scott Sheffield}

\address{
Department of Mathematics, MIT,
77 Massachusetts Avenue,
Cambridge, MA 02139, USA}
\email{sheffield@math.mit.edu}

\author{Wendelin Werner}

\address{Department of Mathematics, ETH Z\"urich, R\"amistr. 101, 8092 Z\"urich, Switzerland} 
\email{wendelin.werner@math.ethz.ch}

\begin{abstract}
We show that when observing the range of a chordal $\SLE_\kappa$ curve for $\kappa \in (4,8)$, it is not possible to recover the order in which the points have been visited.  We also derive related results about conformal loop ensembles (CLE): 
(i) The loops in a $\CLE_\kappa$ for $\kappa \in (4,8)$ are not determined by the $\CLE_\kappa$ gasket.  
(ii) The continuum percolation interfaces defined in the fractal carpets of conformal loop ensembles $\CLE_{\kappa}$ for  $\kappa \in (8/3, 4)$ (we defined these percolation interfaces
in earlier work, where we also showed there that they are $\SLE_{16/\kappa}$ curves) are not determined by the $\CLE_{\kappa}$ carpet that they are defined in. 
\end{abstract}

\date{\today}
\maketitle

\parindent 0 pt
\setlength{\parskip}{0.20cm plus1mm minus1mm}

\section{Introduction}
\label{sec:intro}

The Schramm-Loewner evolutions ($\SLE$) defined by Oded Schramm \cite{S0} in 1999  are the canonical conformally invariant, non-crossing, fractal curves which connect a pair of boundary points in a simply connected planar domain, and their importance has since been highlighted in numerous settings.  The SLE family is indexed by the positive real parameter $\kappa$, and there are three different regimes of $\kappa$ values which correspond to different sample path behavior of an $\SLE_\kappa$ curve  \cite{RS05}: it is a simple curve for $\kappa \in (0,4]$, a self-intersecting but not space-filling curve for $\kappa \in (4,8)$, and a space-filling curve for $\kappa \geq 8$.

In this work, we will answer the following question: Is it possible to recover the \emph{trajectory} $\eta$ of the $\SLE_\kappa$ (i.e.\ the map $t \mapsto \eta(t)$) when one observes its \emph{range} (i.e.\ the set of points $\eta ([0, \infty))$)? In other words, if one knows the \emph{set} of points that an $\SLE_\kappa$ visited, can one recover the \emph{order} in which they were traced?  The answer in the regime where $\kappa \in (4,8)$ is the following:

\begin{theorem}[$\SLE_\kappa$ range does not determine the path] 
\label{thm:path_not_determined}
Fix $\kappa \in (4,8)$, suppose that $\eta$ is an $\SLE_{\kappa}$ process from $0$ to $\infty$ in the upper half-plane, and consider some fixed $T \in (0, \infty]$.  Then the trajectory $\eta|_{[0,T]}$ is almost surely not determined by its range $\eta([0,T])$.  In fact, the conditional law of the trajectory given its range is almost surely non-atomic.
\end{theorem}

We note that the answer to this question in the other regimes of $\kappa$ values is trivial:  For a simple curve one can always recover the trajectory given its range, and  for a space-filling curve, the range does provide no information at all.

In the proof of this result, we will view our $\SLE_\kappa$ path as living in an ambient space with many other $\SLE_\kappa$ paths around.  More specifically, we will take here this space to be a collection of loops which come from a conformal loop ensemble $\CLE_\kappa$  \cite{SHE_CLE,SHE_WER_CLE} with $\kappa \in (4,8)$.  This is in contrast to several other recent works, in which it was natural to use the Gaussian free field (GFF) as the structure in which the path is naturally embedded \cite{SHE_WELD,DUB_PART,MS_IMAG,MS_IMAG2,MS_IMAG3,MS_IMAG4}.

We will also answer two natural questions about $\CLE$. Recall that a $\CLE$ describes the distribution of a natural random collection of loops in a simply connected domain.  The law of a $\CLE$ in a simply connected domain is conformally invariant (and one can therefore always view it as the conformal image of a $\CLE$ defined in the unit disk) and it is described by the same parameter~$\kappa$ as SLE, but with the constraint that $\kappa$ has to be in the interval $(8/3,8)$.  The loops of a $\CLE_\kappa$ are $\SLE_\kappa$-type paths.  Again, there are two ranges depending on whether or not $\kappa > 4$.  When $\kappa \in (8/3, 4]$, the loops are simple, disjoint and do not touch the boundary of the domain they are defined in, whereas when $\kappa \in (4,8)$, the loops are non-simple (but non-self-crossing) and can touch each other and the boundary. 

Our first result about $\CLE$ will deal with the latter case (where $\kappa \in (4,8)$). The set of points that is not surrounded by any of the loops (in the sense that the index of all the loops around those points is $0$) is a random closed set called the $\CLE_\kappa$ gasket, and can be viewed as the natural conformally invariant random version of the Sierpinski gasket. Because the individual loops of the $\CLE_{\kappa}$ touch each other and the boundary,  it is a priori not clear whether one can recover the individual loops by just looking at the gasket.  Indeed, it is not possible: 

\begin{theorem}[$\CLE_{\kappa}$ gasket does not determine $\CLE_{\kappa}$ loops]
\label{thm:cle_not_determined}
Fix $\kappa \in (4,8)$ and suppose that $\Gamma$ is the collection of loops in a (non-nested) $\CLE_{\kappa}$.  Then $\Gamma$ is almost surely not determined by the $\CLE_{\kappa}$ gasket.  In fact, the conditional law of $\Gamma$ given its gasket is almost surely non-atomic.
\end{theorem}

We now turn to our second result for $\CLE$, which is focused on the case of $\CLE_\kappa$ for $\kappa \in (8/3, 4)$. Recall that in this case, the loops in the $\CLE_\kappa$ form a disjoint collection of simple loops that also do not touch the boundary, so that the set of points that are encircled by no loop (this set is now called the $\CLE_\kappa$ carpet) can be viewed as a natural conformally invariant random version of the Sierpinski carpet. In \cite{cle_percolations}, we have defined and described natural continuous percolation interfaces (CPI) within such $\CLE_\kappa$ carpets, that can intuitively describe boundaries of critical percolation clusters within these random fractal sets. These interfaces turn out to be variants of $\SLE_{16/\kappa}$ curves, that are coupled with the $\CLE_\kappa$.

\begin{theorem}[Continuous percolation within $\CLE_\kappa$ is random]
\label{thm:percolation_not_determined}
Fix $\kappa \in (8/3,4)$, suppose that $\Gamma$ is a $\CLE_\kappa$, and that $\eta$ is an $\SLE_{16/\kappa}$-type curve coupled with $\Gamma$ as a CPI in the sense of \cite{cle_percolations}.  Then the range of $\eta$ (and therefore also the path) is almost surely not determined by $\Gamma$. In fact, the conditional law of $\eta$ given $\Gamma$ is almost surely non-atomic. 
\end{theorem}

In fact, this statement also holds in the case of the ``labeled'' $\CLE_\kappa$ for $\kappa \in (8/3, 4)$ which are described in \cite{cle_percolations}, where for each of the $\CLE_{\kappa}$ loops, one tosses an independent biased coin to decide whether it is open or closed for the considered percolation process that one constructs. We note that the analog of Theorem~\ref{thm:percolation_not_determined} for the labeled $\CLE_4$ is known to be false (see \cite{cle_percolations}): the continuous percolation interfaces in a labeled balanced (one uses a fair coin to choose the labels) $\CLE_4$ are deterministic functions of the labeled $\CLE_4$ itself. 

Theorem~\ref{thm:percolation_not_determined} also sheds some light about the coupling between the GFF and the $\CLE_\kappa$ carpets when $\kappa \in (8/3, 4)$, as it shows that in these couplings, the GFF is not a deterministic function of the nested $\CLE_\kappa$ carpets.

\begin{figure}[ht!]
\includegraphics[scale=1.2]{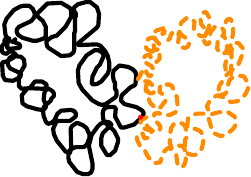} \quad
\includegraphics[scale=1.2]{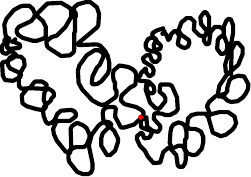}
	\caption{\label{fig:merging_loops}Idea of the proof of Theorem~\ref{thm:cle_not_determined}:  Choosing a pivotal and resampling its state can merge loops without changing the gasket.}
\end{figure}

\begin{figure}[ht!]
\includegraphics[scale=0.8]{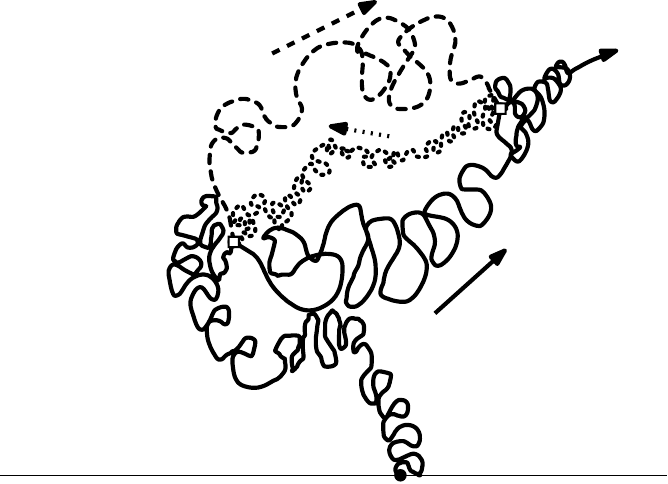}
	\caption{\label{fig:pair_of_pivotals}Idea of the proof of Theorem~\ref{thm:path_not_determined}: An $\SLE_{\kappa}$ path with $\kappa \in (4,8)$ and two intertwined double points. The path depends on whether it visits the plain part before the dashed part or not, but the range does not. Choosing two such double points according to some well-chosen measure $\mu'$ and resampling the order between dashed and plain will preserve the range but not the path}.
\end{figure}

The rough idea of the proof of Theorem~\ref{thm:cle_not_determined} will be to construct a measure $\mu$ which is supported on a certain set of exceptional points.  These points are either intersection points between two distinct macroscopic CLE loops, or double points on one single loop.  In both cases, there are four different macroscopic strands that emanate from these points. We will then show that if one performs the Markov step of picking a point at random using $\mu$, and then resamples the way that the four macroscopic strands are hooked up at that point, one roughly preserves the law of the global $\CLE$. During this resampling step, the gasket is preserved, but one can merge two loops into one or split one loop into two (see Figure~\ref{fig:merging_loops}). This shows that it is not possible to identify the individual loops by just observing the gasket.

The proofs of Theorem~\ref{thm:percolation_not_determined} and of Theorem~\ref{thm:path_not_determined} will follow a similar idea, except that in the latter case, one will need to construct a measure $\mu'$ on special \emph{pairs} of intertwined double points on the path, and the Markov step will consist in switching simultaneously the hookup configuration between the four strands at both points in order to preserve the range of the path (see Figure~\ref{fig:pair_of_pivotals}).  This then leads to a coupling of a pair of $\SLE_{\kappa}$ paths which have the same range but visit their common range in a different order.

The construction and non-triviality of the measures $\mu$ and $\mu'$ is based on the multi-scale second moment method, which has also been used in many instances in the last decades to study the Hausdorff dimensions of random fractals.

\section{Preliminaries}
\label{sec:prelim}

In this section, we make some comments and review some preliminary facts before proceeding to the proofs of our main theorems. 
In Section~\ref{subsec:percolation_analogy}, we will explain a possible approach to proving our results when $\kappa=6$ using properties of Bernoulli percolation.  In Sections~\ref{subseckapparho}--\ref{subsec:trunk_construction}, we will review the construction of $\CLE_\kappa$ when $\kappa \in (4,8)$ as well as the construction when $\kappa \in (8/3,4]$ given in \cite{cle_percolations}.

\subsection{Percolation pivotal points analogy}
\label{subsec:percolation_analogy}

Although it will not be used directly in our proofs, it is worthwhile to first describe a possible proof of Theorem~\ref{thm:path_not_determined} in the special case where $\kappa=6$. Indeed, our proof will have some analogies with the strategy that we will now outline.

The $\SLE_6$ curve is known to be the scaling limit of percolation interfaces and the way in which the discrete interface approaches the $\SLE_6$ in the scaling limit is well-understood \cite{S01,CN06,WWlnperco}. In particular, the double points of $\SLE_6$ correspond to the scaling limit of the double points in the discrete percolation interface (see for instance 
\cite{WWlnperco}). These discrete double points form a subset of the set of points in the percolation configuration where a so-called four-arm event holds (two disjoint closed and two open arms touch these points in alternating circular order, and they create four percolation interface strands). Furthermore, thanks to the work of Garban, Pete and Schramm \cite{GPS1,GPS2}, the way in which the counting measure on such double points approaches a natural measure on the set of double points of $\SLE_6$ in the scaling limit is well-understood.

Suppose now that one considers a long percolation interface, and consider two given disjoint macroscopic domains (here the domains will be thought of as fixed, while the mesh of the lattice will tend to zero). With positive probability, it will happen that the percolation interface visits these two domains twice and create intertwined double points as in Figure~\ref{intertwined2}.  On this event, we can decide to choose at random a pair of such intertwined double points using the counting measure on such pairs, and to change the status of both of these two points simultaneously. Note that this will basically not change the range of the percolation interface, but only the order in which the three strands of the percolation interface that join the two double points are traced. This indicates that the probability of the obtained configuration is comparable to the probability of the initial one (before switching how the strands are hooked up), which in turn indicates that in the scaling limit, the $\SLE_6$ 
curve cannot be a deterministic function of its range.

Let us note that in order to make the previous idea work, it is sufficient to use the counting measure on some ``special'' intertwined double points which satisfy some extra conditions. For instance, one can  use points where the four strands are well-separated at some macroscopic scale. Such points are easier to work with when obtaining uniform estimates: In particular, when one conditions on the event that such a well-separated four-arm event occurs at two given points, then one can see that the conditional probability that these two points end up being intertwined on the percolation interface is bounded from below. Instead of sampling according to the counting measure along the curve, one can (up to constants in the probabilities) therefore first choose the two points at random in some uniform way in two prescribed domains, then sample the nice four-arm events in their respective macroscopic neighborhood, and then finally the percolation configuration in the remaining domains, in such a way that they hook 
up the arms so that the percolation interface visits these two points, and then finally the state of these two  points.

\begin{figure}[ht!]
\includegraphics[scale=0.6]{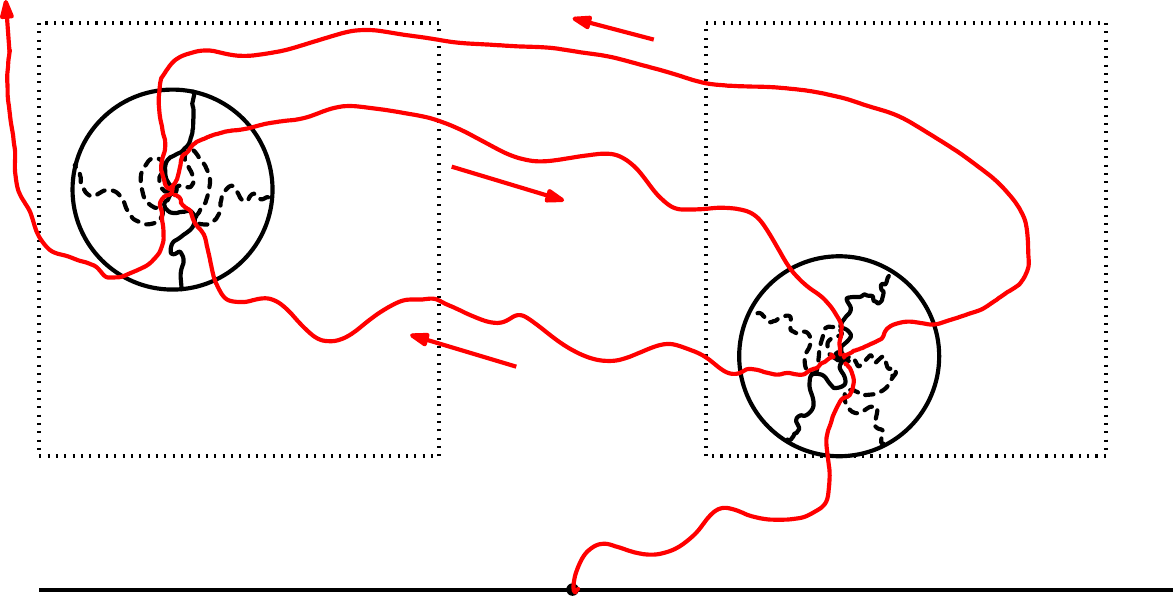}
	\caption{Choose two points in the dotted regions and condition on the independent nice four-arm events for both of them. Then, the conditional probability of the interface creating the intertwined double points as shown in the figure is bounded from below, and the order that the three strands are visited changes if one changes the way in which the strands are hooked up at those 
	double points. \label{intertwined2}}
\end{figure}

Our approach to the general $\SLE_{\kappa}$ case will have a similar flavor, although we will work directly in the continuum. The percolation configuration will be replaced by a $\CLE_\kappa$ instance.  We will discover the $\CLE_\kappa$ first near the two points $z,w$ using branches of the $\CLE_\kappa$ exploration tree.  We will also define some ``nice four-arm type events'' and use the conformal Markov property of $\SLE_\kappa$ and $\CLE_\kappa$  to control the correlations between what happens in different regions.

\subsection {Background on $\SLE_\kappa (\kappa -6)$ and $\SLE_\kappa (\rho; \kappa -6 - \rho )$ processes} 
\label {subseckapparho}

Certain generalizations of SLE processes will play in important role in our paper, and we briefly review their definition and basic properties here.

\subsubsection*{Loewner chains and SLE} 
Recall that Loewner's equation allows one to define a increasing family of compact hulls $(K_t)_{t \ge 0}$ in the closed upper half-plane $\overline \HH$,
when one is given a continuous real-valued function $W$, using the following procedure: 
Define for each $z \in \overline \HH$ the solution $(g_t (z))_{t \le \tau (z)}$ to the ordinary differential equation 
\[ \partial g_t (z) = \frac {2} { g_t (z) - W_t } \]
started from $g_0 (z)= z$, up to the (possibly infinite) time $\tau(z)$ which is the first time at which $g_t (z) - W_t$ hits $0$. The set $K_t$ is then defined to be $\{ z \in \overline \HH, \tau (z) \le t \}$ and $g_t$ turns out to be the unique conformal transformation from $H_t := \HH \setminus K_t$ into $\HH$ such that $g_t (z) - z = o(1)$ as $z \to \infty$.  
The family $(K_t)_{t \ge 0}$ is the Loewner chain driven by the function $W$. 

In all cases that we will be dealing with, it turns out that there exists a continuous two-dimensional curve $\gamma$ with the property that at all times $t$, $H_t$ is the unbounded connected component of $\HH \setminus \gamma [0,t]$. One then says that the Loewner chain is generated by this curve $\gamma$. This curve is then also uniquely determined by the driving function $W$. Note that the existence of such a curve $\gamma$ does not hold for all driving functions $W$, but it is actually now known to hold for all the random driving functions $W$ that we will discuss in the present paper (but proving this has been a quite challenging endeavor). 

\medbreak 

When $\kappa$ is some positive constant, and $(B_t)_{t \ge 0}$ is a standard one-dimensional Brownian motion, the Loewner chain obtained when $W_t := \sqrt {\kappa} B_t $ is Schramm's $\SLE_\kappa$ 
\cite {S0}. In the present paper, we will actually consider generalization of the SLE curves only for $\kappa \in (2, 8)$ (so, we will omit to describe here what happens for $\kappa \ge 8$). 
One can recall \cite {RS05} that $\SLE_\kappa$ is almost surely generated by a simple continuous curve when $\kappa \in (0, 4]$ and generated by a continuous curve with a dense family of double points when $\kappa \in (4, 8)$.
The scaling property of Brownian motion ensures that $\SLE_\kappa$ is conformally invariant and can be defined in any simply connected domain with two marked boundary points (up to time-reparameterization of the curve).
The Markov property of Brownian motion implies that the SLE possesses its important conformal Markov property, see \cite {S0}.

\subsubsection*{$\SLE_\kappa (\rho_1 ; \rho_2)$ for $\rho_1, \rho_2 > -2$}
Let us now describe the $\SLE_\kappa (\rho_1 ; \rho_2)$ processes when $\rho_1 , \rho_2 > -2$ that will play an important role in the present paper. These are the natural generalizations of SLE where one keeps track of two additional marked boundary points. Suppose that $\kappa \in (0, 8)$. 
The $\SLE_\kappa (\rho_1 ; \rho_2)$ from $0$ to infinity in the upper half-plane is the Loewner chain driven by the random function $W$ that is defined by 
\[ dW_t = \sqrt{\kappa} dB_t + \frac{\rho_1}{W_t - V_t^1} dt +  \frac{\rho_2}{W_t - V_t^2} dt, \quad dV_t^j = \frac{2}{V_t^j - W_t} dt, \quad V_0^j = 0 \quad\text{ for }  \quad j=1,2,\]
with the condition that $V_t^1 \le W_t \le V_t^2$ for all $t \ge 0$ (we sometimes say that $V_0^1 = 0_-$ and $V_0^2 = 0_+$ to indicate this).  
When $\rho_1 =0$ or $\rho_2 =0$, these are the usual $\SLE_\kappa (\rho)$ processes with just one additional marked point that had been introduced in \cite {LSW04}. The fact that these $\SLE_\kappa (\rho_1; \rho_2)$ are 
well-defined and generated by continuous curves for all these choices of $\rho_1, \rho_2 > -2$ has been derived in \cite{MS_IMAG}. Again, because the driving function $W$ of these processes 
satisfies Brownian scaling, the obtained Loewner chain and curve are conformally invariant, so that it is for instance possible to define an $\SLE_\kappa (\rho_1 ; \rho_2)$ from $0$ to $x \in \R \setminus \{ 0 \}$ in the upper half-plane as the image of the previous curve $\gamma$ under a conformal automorphism $\phi$ of $\HH$ with $\phi (0)=0$ and $\phi (\infty) = x$.  

When $\kappa \in (2, 8)$ and $\rho \in (-2, \kappa -4)$, then the $\SLE_\kappa ( \rho ; \kappa -6 - \rho)$ turns out to have the very special target-invariance property (see \cite{dub2007commutation,SCHRAMM_WILSON} for the case 
where $\rho = 0$): Suppose that for each $y$, 
${\mathcal L}_y$ is the law of an $\SLE_\kappa 
(\rho ; \kappa -6 - \rho)$ curve from $0$ to $y \in \R \setminus \{ 0 \}$ in the upper half-plane. When $\gamma_y$ is such a curve, we  let $\tau_y (y')$ be the first time at which $\gamma_y$ does disconnect $y'$ from $y$. 

\begin {definition}[Target-invariance] 
We say that the family $({\mathcal L}_y)_{y \in \R \setminus \{ 0 \} }$ has the target-invariance property if for any $y \not= y'$ in $ \R \setminus \{ 0 \}$, it is possible to couple a realization $\gamma_y$ 
of ${\mathcal L}_y$ and a realization $\gamma_{y'}$ of ${\mathcal L}_{y'}$ in such a way 
that $\gamma_y$ up to $\tau_y (y')$ and $\gamma_{y'}$ up to $\tau_{y'} (y)$ coincide almost surely (up to time-reparameterization). 
\end {definition} 

For the target-invariant family $({\mathcal L}_y)$ defined above, it is then possible to actually couple all the paths $\gamma_y$ for $y \in \R \setminus \{ 0 \}$ on the same probability space, 
in such a way that for all $y \not= y'$ in $\R \setminus \{ 0 \}$, the paths $\gamma_y$ up to $\tau_y (y')$ and $\gamma_{y'}$ up to $\tau_{y'} (y)$ coincide almost surely (up to time-reparameterization) -- one can for instance 
first do this for any given countable family of target points $y$, and then to invoke a Kolmogorov extension result: 

\begin {definition}[Branching tree] 
This path-valued process $( \gamma_y)_{y \in \overline \R \setminus \{ 0 \}}$ is called the $\SLE_\kappa (\rho ; \kappa - 6 - \rho)$ branching tree in $\HH$, rooted at $0$ and targeting all boundary points.  
\end {definition}
Note that we view this here as a path-valued process, and that we do not discuss here the ``regularity'' with respect to $y$. All events that we will discuss will 
in fact depend on a given countable  (say, dense) family of paths $\gamma_{y_1},\gamma_{y_2}, \ldots$ and there will be no measurability issue for those. 

One particular case is when $\kappa \in (4, 8)$ and $\rho = 0$ (or $\rho = \kappa - 6$). One then obtains the $\SLE_\kappa ( \kappa -6)$ branching tree, which (as we will recall)
is then used to define $\CLE_\kappa$ for $\kappa \in (4,8)$. In that case, the target-invariance property and the branching tree are actually naturally defined for a richer class of 
$\SLE_\kappa (\kappa -6)$, that 
target also any interior point $y \in \HH \setminus \{ 0 \}$ (instead of only boundary points).

\subsubsection*{$\SLE_\kappa^\beta (\kappa -6)$ for $\kappa \in (8/3, 4)$}
Let us now briefly recall some features of the $\SLE_\kappa (\kappa -6)$ processes when $\kappa \in (8/3, 4)$. Care and some non-trivial observations are required to define 
these processes, because in this case $\kappa - 6 < -2 $, so that the drifts can become singular (i.e., not absolutely integrable) when one tries to   
directly generalize the previous definition of $\SLE_\kappa (\rho)$ processes via
\[ dW_t = \sqrt{\kappa} dB_t +   \frac{\rho}{W_t - V_t} dt, \quad dV_t = \frac{2}{V_t - W_t} dt, \quad V_0 = 0 \]
to these cases (the main observation is that formally, the process 
$((W_t - V_t ) / \sqrt {\kappa})$ would then be a Bessel process of small dimension, so that $\int ds / | W_s - V_s| = \infty$ in the neighborhood 
of the times at which $W_t - V_t =0$).

It turns out that (we refer to Section 3 of \cite {cle_percolations} for a detailed self-contained presentation and references to the original papers) it is still possible 
to make sense of $\SLE_\kappa ( \kappa - 6)$ processes in the range $\kappa \in (8/3, 4)$: 
\begin {itemize} 
 \item One can define a $\SLE_\kappa^1 (\kappa -6)$ process and a $\SLE_\kappa^{-1} (\kappa -6)$ process, which have the property that the sign of $W_t - V_t$ remains constant (i.e., non-negative or non-positive).  
 These are the totally asymmetric $\SLE_\kappa (\kappa -6)$ processes. 
 \item For general $\beta \in (-1, 1)$, one can define the side-swapping $\SLE_\kappa^\beta (\kappa -6)$ process, where (loosely speaking) the sign of each excursion of the process $W - V$ away from the origin is chosen 
 independently to be $+$ or $-1$ with respective probability $(1 \pm \beta)/2$.
\end {itemize}
The scale-invariance of these processes (and therefore the fact that one can define them in any simply connected domain) is basically part of their construction. 
The continuity of all these $\SLE_\kappa^\beta (\kappa -6)$ curves was a challenging question and has been established in \cite{cle_percolations}. 
Their target-invariance property is basically part of their definition, and this enables 
to define (for each value of $\beta$) the $\SLE_\kappa^ \beta (\kappa -6)$ branching trees (that we will use in the present paper, and that provide one way to define conformal loop ensembles for $\kappa \in (8/3, 4)$). 

One can emphasize here that for these $\SLE_\kappa^\beta (\kappa -6)$, it is 
also possible (and easy) to define the target-invariance property and the branching tree for a collection of paths $(\gamma_y)_{y \in \overline \HH \setminus \{ 0 \}}$ that target also any 
interior point $y \in \HH$.  

\subsection{$\CLE_{\kappa}$ background when $\kappa \in (4,8)$}
\label{subsec:cle_bacground}

We will now give a very brief review of the construction and the main properties of $\CLE_{\kappa}$ for $\kappa \in (4, 8)$; mind that some of the statements that we will survey here do {\em not} hold for $\kappa \in (8/3, 4]$. These results follow fairly directly from \cite{SHE_CLE} -- see also \cite{cle_percolations,MSW_CLE_GASKET} for more extensive reviews.  
In the present paper, we will focus primarily on the non-nested versions of conformal loop ensembles, where each given point in the domain is almost surely surrounded by one CLE loop. 

Non-nested $\CLE_\kappa$ for $\kappa \in (4,8)$ are random collections of loops, that are the conjectural scaling limits of the collection of outermost 
interfaces for a critical FK model for $q=2+2\cos(8\pi/\kappa) \in (0,4)$ with free boundary conditions on a deterministic planar lattice approximation of a simply connected domain.  The special cases $\kappa=6$ and $\kappa=16/3$ have been shown to correspond respectively 
to the scaling limits of site percolation on the triangular lattice \cite{S01,CN06} (as mentioned in the previous subsection) and of the critical Ising-FK model \cite{S07}.  One can also relate the $\CLE_\kappa$ with these FK-models in the framework on planar maps and Liouville quantum gravity:  
Combining the results of \cite{she2011qg_inv,dms2014mating}  implies that $\CLE_\kappa$ is the scaling limit (for the so-called peanosphere-topology) of the interfaces in the critical FK model for $q=2+\cos(8\pi/\kappa)$ on certain types of random planar map models. 

Let us now recall the definition of the $\CLE_\kappa$ exploration tree from  \cite{SHE_CLE} -- we still assume that $\kappa \in (4,8)$ here.  One starts with the $\SLE_\kappa (\kappa-6)$ branching tree as described in the previous subsection (we consider here the one in the upper half-plane, rooted at $0$ and targeting all points in the closed upper half-plane).  Suppose that $D \subseteq \C$ is a non-trivial simply connected domain and that one chooses a  starting point (or root) $x$ on $\partial D$. By mapping the upper half-plane onto $D$ (and the origin onto $x$), we obtain the $\SLE_\kappa (\kappa -6)$ branching tree in $D$, rooted at $x$.

Then, using this tree, and guided by the conjectures about discrete models, it is explained in \cite{SHE_CLE} how to define a collection of loops.  For a given point $z$, let $\eta_z$ be the branch of the exploration tree targeted at $z$ and let $\tau_z$ be the first time that $\eta_z$ surrounds $z$ clockwise.  In other words, $\tau_z$ is the first time $t$ that the harmonic measure of the right side of $\eta_z([0,t])$ as seen from $z$ is $1$.  Let $\sigma_z$ be the largest time $t$ before $\tau_z$ that the harmonic measure of the right side of $\eta_z([0,t])$ as seen from $z$ is $0$.  Then the (outermost) loop surrounding $z$ is given by the concatenation of $\eta_z|_{[\sigma_z,\tau_z]}$ with the branch of the exploration tree from $\eta_z(\tau_z)$ to $\eta_z(\sigma_z)$.  We can for instance consider the collection of all loops that surround points $z$ with rational coordinates.  This defines the countable collection of loops that form the  $\CLE_{\kappa}$, and gives rise to a number of natural conjectures \cite {SHE_CLE} about this object. 

Let us list a few properties of these $\CLE_\kappa$'s for $\kappa \in (4,8)$ that have been derived using the imaginary geometry approach to SLE processes: 
\begin {itemize}
\item  It was conjectured in \cite{SHE_CLE} that these loops are in fact continuous curves.  This is now known to hold because (see \cite{MS_IMAG})
 $\SLE_\kappa(\rho)$ curves have been proved to be continuous when $\rho > -2$ (and when $\kappa >4$, then $\kappa-6 > -2$).
\item 
It was conjectured in \cite{SHE_CLE} (this is very natural because this property holds in the case of discrete models) that the law of this collection of loops does not depend on the choice of the root of the exploration tree. This property does not follow trivially from the branching tree definition and setup, but it has now been established, using the reversibility properties of $\SLE_{\kappa} (\kappa -6)$ for $\kappa \in (4,8)$ derived in \cite{MS_IMAG3}. 
\item 
The local finiteness of $\CLE_\kappa$, i.e.\ that the number of loops with diameter at least $\epsilon$ is for each $\epsilon > 0$ almost surely finite, was established in \cite{MS_IMAG4} as a consequence of the almost sure continuity of the so-called space-filling $\SLE$. 
\item
The $\CLE_\kappa$ is clearly a deterministic function of the exploration tree.  Conversely, as explained in \cite{SHE_CLE}, when $\kappa \in (4,8)$, the exploration tree are in fact a deterministic function of the $\CLE_{\kappa}$ and the chosen root.  This makes it possible to discover simultaneously different portions of
different exploration trees starting from different roots but that are associated with the same $\CLE_\kappa$. This idea will play a key role in the present paper.
\end {itemize}

\subsubsection*{$\CLE_\kappa$ background when $\kappa \in (8/3,4)$ and continuum percolation construction}
\label{subsec:percolation_review}

Let us first say some very brief words about the basic properties of $\CLE_\kappa$ when $\kappa \in (8/3, 4)$. In the present paper, we will actually only use their 
construction via boundary conformal loop ensembles (BCLE) but we need to recall a few things about them first. 

\begin {itemize}
 \item Just as for $\kappa \in (4,8)$, one chooses a simply connected domain $D$ and a boundary point $x$, and then considers the $\SLE_\kappa^\beta (\kappa -6)$ branching tree rooted at $x$.  (Recall that one needs here to use side-swapping and/or L\'evy compensation because $\kappa - 6 < -2$. So, when $\kappa \in (8/3, 4)$,
 one has to first choose a side-swapping parameter $\beta \in [-1,1]$ to define such a tree -- and when $\kappa =4$, one would have to choose a drift parameter $\mu$.)
 \item The fact that the law of $\CLE_\kappa$ that is constructed in this way does not depend on the root is non-trivial, and does rely on another construction of these 
 CLEs using the Brownian loop-soups (see \cite{SHE_WER_CLE}). Actually, the results of the paper \cite{cle_percolations} that we will recall in a few paragraphs do 
 provide as a by-product another derivation of the root-independence of the $\CLE_\kappa$ distribution, based on the coupling of $\SLE$ with the GFF. The fact that the law of the $\CLE_\kappa$ does also not depend 
 on the choice of $\beta$ (or $\mu$ when $\kappa =4$) is explained in \cite{ww2013conformally}. So in summary, there is indeed just one $\CLE_\kappa$ distribution for each $\kappa \in (8/3, 4]$. 
 \item It is very important to stress that it has not been proved that in this construction, the exploration trees that define these $\CLE_\kappa$'s can be 
 recovered deterministically from the $\CLE_\kappa$ and the root. In fact, Theorem~\ref{thm:percolation_not_determined} of the present paper will show that in the case $\kappa \in (8/3,4)$,  the $\CLE_\kappa$ exploration tree is \emph{not} a deterministic function of the $\CLE_\kappa$ and the root (for the special case $\kappa=4$,
 we refer to the discussion in \cite{cle_percolations}: The exploration defines a tree only for the balanced labeled $\CLE_4$, which together with the choice of the root, does determine the exploration tree -- as can be shown using the direct relationship between this labeled $\CLE_4$ and the GFF). 
\end {itemize}

\subsubsection* {Boundary conformal loop ensembles} 

We now recall some further features from the paper \cite{cle_percolations} that will be relevant in the present paper:
Suppose now that $\kappa \in (2,4)$ and $\rho \in (-2,\kappa-4)$.  

Let us first recall the construction of the so-called \emph{boundary conformal loop ensembles} $\BCLE_\kappa(\rho)$ defined in \cite[Section~7]{cle_percolations}.  This is a conformally invariant family of boundary-touching $\SLE_\kappa$-type loops which live in a simply connected domain $D$. Despite the fact that $\kappa  < 4$, its definition does follow rather closely that of non-simple CLEs that we recalled in the previous subsection.  As we will here sometimes describe simultaneously some SLE processes for different values of $\kappa$, we will use in this subsection the notation $\kappa \in (2, 4)$ and $\kappa' = 16/ \kappa \in (4, 8)$, as in \cite{cle_percolations}. 

We fix a root point $x \in \partial D$.  Consider now the $\SLE_\kappa (\rho; \kappa -6 -\rho)$ branching tree rooted at $x$ and targeting all points $y \in \partial D \setminus \{ x \}$ 
(mind that this tree targets only boundary points).  
The union of all branches in this tree divides $D$ into a countable collection of connected components. The boundary of each connected component is naturally oriented by the paths of the tree which form its boundary.  The collection of boundaries of subdomains that are naturally oriented clockwise are called the loops of 
the $\BCLE_\kappa(\rho)$ and the counterclockwise ones are referred to as the false loops of the $\BCLE_\kappa(\rho)$.  If we want to emphasize that the loops have a clockwise orientation, we will use the notation $\cwBCLE_\kappa(\rho)$.  One similarly defines $\ccwBCLE_\kappa(\rho)$ using $\SLE_\kappa(\kappa-6-\rho;\rho)$ in place of $\SLE_\kappa(\rho;\kappa-6-\rho)$ and takes the loops (resp.\ false loops) which are traced counterclockwise (resp.\ clockwise).  
Again, it is not obvious from the construction, but it is shown in \cite[Proposition~7.1]{cle_percolations} (using the reversibility of the $\SLE_\kappa(\rho_1;\rho_2)$ processes with $\rho_1,\rho_2 > -2$ established in \cite{MS_IMAG2}) that these $\BCLE_\kappa (\rho)$ do not depend on the choice of root $x$.

For $\kappa' \in (4,8)$ and $\rho' \in (\kappa'/2-4,\kappa'/2-2)$, the $\cwBCLE_{\kappa'}(\rho')$ and $\ccwBCLE_{\kappa'}(\rho')$ are defined in an analogous way and it follows from the reversibility of $\SLE_{\kappa'}(\rho_1';\rho_2')$ with $\rho_1',\rho_2' \geq \kappa'/2-4$ established in \cite{MS_IMAG3} that the resulting family of loops does not depend on the choice of root \cite[Proposition~7.1]{cle_percolations}.  Recall that $\kappa'/2-2$ is the threshold below which the $\SLE_{\kappa'}(\rho')$ processes are boundary intersecting.  The range of $\rho'$ values considered in the definition of $\BCLE_{\kappa'}(\rho')$ is precisely so that $\rho' < \kappa'/2-2$ and $\kappa'-6-\rho' < \kappa'/2-2$ so that the loops do in fact hit the domain boundary.  Note that $\BCLE_{\kappa'}(0)$ is simply the collection of loops in a $\CLE_{\kappa'}$ which intersect the boundary.

We now suppose again that $\kappa \in (8/3, 4)$ (so that $\kappa' \in (4,6)$).  As explained in \cite[Section~7.2]{cle_percolations}, one can iterate $\BCLE$s, alternating between $\kappa$ and $\kappa'$ loops in order to produce natural couplings of $\CLE_\kappa$ and $\CLE_{\kappa'}$.  The construction proceeds as follows.

\begin{itemize}
\item Sample a $\cwBCLE_{\kappa'}(0)$ process $\Gamma'$.  Sampling these loops is the continuum analog of exploring the boundary-touching FK clusters with free boundary conditions.
\item Given $\Gamma'$, we then sample independent $\ccwBCLE_\kappa(-\kappa/2)$ processes in each of the (clockwise) loops of $\Gamma'$.  Call the resulting collection of loops $\Gamma$.  Sampling these loops is the continuum analog of exploring the boundary touching interfaces in the corresponding Potts model with free boundary conditions.
\item Iterate the exploration in the false loops of $\Gamma'$ and $\Gamma$.
\end{itemize}

It is shown in \cite[Theorems~7.2 and~7.3]{cle_percolations} that the law of the collection of $\SLE_\kappa$-type loops thus defined is in fact a $\CLE_\kappa$ (this therefore provides a construction of $\CLE_\kappa$ that does not 
rely on any $\SLE_\kappa (\rho)$ process for $\rho < -2$). 
Note that the proof in \cite {cle_percolations}  uses the $\SLE$ commutation relations which are encoded by the GFF \cite{MS_IMAG,MS_IMAG4}. 

The branch of the $\CLE_{\kappa'}$ exploration tree in this construction is called 
a \emph{continuum percolation exploration} (CPI) inside of the $\CLE_\kappa$ carpet (see \cite[Definition~2.1]{cle_percolations} as well as \cite[Section~4]{cle_percolations}) because it can be interpreted as a percolation interface within this $\CLE_\kappa$ carpet. 
This $\CLE_{\kappa'}$/$\CLE_\kappa$ coupling derived in \cite {cle_percolations} is  a continuum version of the random cluster representation of the Potts model (see, e.g., \cite{GRIM_CLUSTER_BOOK} for a review). See also \cite{BEN_HONG} for the case of the Ising model.

Note that in the  $\BCLE_{\kappa'}$/$\BCLE_\kappa$ iteration scheme described just above, the $\BCLE_\kappa$ loops are always attached to the right side of the $\BCLE_{\kappa'}$ loops.  This means that the CPI always reflects to the left whenever it hits a $\CLE_\kappa$ loop (in the percolation interpretation, 
the interiors of the $\CLE_\kappa$ loops are closed, and the CPI traces the open/closed interface, with closed to the right and open to the left of the interface). 

This construction can be generalized to the setting in which each of the $\CLE_\kappa$ holes is labeled either $+$ or $-$ independently with a given probability $p \in (0,1)$ (which is reminiscent of the $\SLE_\kappa (\kappa -6)$ side-swapping mechanism). 
Then, the CPI reflects to the left (resp.\ right) when it hits a loop with a~$+$ (resp.\ $-$) label.  
In this case the CPI is a branch in a $\BCLE_{\kappa'}(\rho')$ exploration tree where~$\rho'$ is a function of~$p$. This provides for each choice of $p$, a different 
$\BCLE_{\kappa'}$/$\BCLE_\kappa$-type iteration scheme that construct a $\CLE_\kappa$ (see \cite {cle_percolations} for all this).

\subsection{The trunk construction of $\SLE_\kappa(\kappa-6)$ for $\kappa \in (8/3,4)$}
\label{subsec:trunk_construction}

In Section~\ref{subsec:percolation_review}, we recalled the construction of $\CLE_\kappa$ which is based on iterated $\B$\CLE's from \cite{cle_percolations}.  We will now describe the analogous procedure where one builds just one $\SLE_\kappa(\kappa-6)$ process for $\kappa \in (8/3,4]$, which is a single branch of the $\SLE_\kappa (\kappa -6)$ exploration tree. 

We recall (see \cite{cle_percolations} and the references therein for additional background) that when $\kappa \in (8/3, 4)$ so that $\kappa-6 < -2$, for each choice of $\beta \in [-1, 1]$, one can define an $\SLE_\kappa^\beta(\kappa -6)$ process so that the following is true.  The value $p =(1+\beta) / 2$ represents the probability that when it traces a loop, it traces it counterclockwise and the trunk passes to the left of that loop while when this process traces a loop clockwise, its trunk passes to the right of that loop. When $\kappa =4$ so that $\rho=-2$, one has to use a symmetric side-swapping (i.e., $\beta=0$ and $p=1/2$ in the previous setup), but there is an additional drift-type parameter $\mu$.  This leads to the so-called $\SLE_4^{0, \mu}$ processes.  All of these processes are called conformally invariant explorations of the $\CLE_\kappa$ that they construct.

We will now describe in a bit more detail the case where $\kappa \not=4$ (i.e., $\kappa \in (8/3,4)$) and for simplicity we will again focus on the case $\beta = 1$ (so that all of the loops are attached to the right side of the trunk).  We suppose that we have a simply connected domain $D \subseteq \C$ and fix $x,y \in \partial D$ distinct.  We then carry out the following steps (see the right-hand side of Figure \ref{fig:sle_bcle}). Here $\rho'=0$ because we are dealing with the case $\beta=1$. 
\begin{itemize}
\item Sample an $\SLE_{\kappa'}(\kappa'-6)$ process $\eta'$ from $x$ to $y$ with the force point located at $x^+$ (i.e., infinitesimally on the counterclockwise side of $x$).  Note that the time-reversal of $\eta'$ is an $\SLE_{\kappa'}(\kappa'-6)$ from $y$ to $x$ with the forced point located at $y^-$.  In other words, the law of $\eta'$ is reversible up to swapping the force point from the right to the left side.  This  process $\eta'$ will end up being the \emph{trunk} of an $\SLE_\kappa^1(\kappa -6)$ from $x$ to $y$ that we will construct. 
\item We next trace the collection of $\CLE_\kappa$-loops that are attached to the trunk $\eta'$ using the following procedure.  In each component $U$ in the complement of the range of $\eta'$ which is either to the right of $\eta'$ or surrounded clockwise by $\eta'$ we sample a collection of $\SLE_\kappa$-type loops as follows.  Let $x_U$ (resp.\ $y_U$) be the first (resp.\ last) point on $\partial U$ which is visited by $\eta'$.  We then draw a branching $\SLE_\kappa(3\kappa/2-6)$ process which starts from $x_U$ and is targeted at every point on $\partial U$ which is in the range of $\eta'$.  When viewed as targeting $y_U$, this process is an $\SLE_\kappa(3\kappa/2-6)$ process with a single force point at $x_U^-$ (when viewed as targeting another point $z$ in $\partial U$ which is also in the range of $\eta'$, this process is an $\SLE_\kappa(3\kappa/2-6;-\kappa/2)$ process with force points at $x_U^-$ and $y_U$); note that $3\kappa/2-6 = \kappa-6-(-\kappa/2)$.  This process will trace a collection of $\SLE_
\kappa$-type loops.  We note that in the case that $x_U = y_U$, this collection of loops has the same law as for a $\BCLE_\kappa(-\kappa/2)$.  In the case that $x_U \neq y_U$, we will still refer to this collection of loops as a $\BCLE_\kappa(-\kappa/2)$ but with the marked boundary points $x_U$  and $y_U$.
 \item We can the construct a continuous path $\eta$ from $x$ to $y$ as follows.  It moves along the trunk $\eta'$ from $x$ to $y$, and each time it meets one of the $\SLE_\kappa$-type loops for the first time, it traces it counterclockwise.  As explained in \cite{cle_percolations}, the law of $\eta$ is that of an $\SLE_\kappa^1(\kappa-6)$ in $D$ from $x$ to $y$. 
\end{itemize}

\begin{figure}[ht!]
\begin{center}
\includegraphics[scale=0.85]{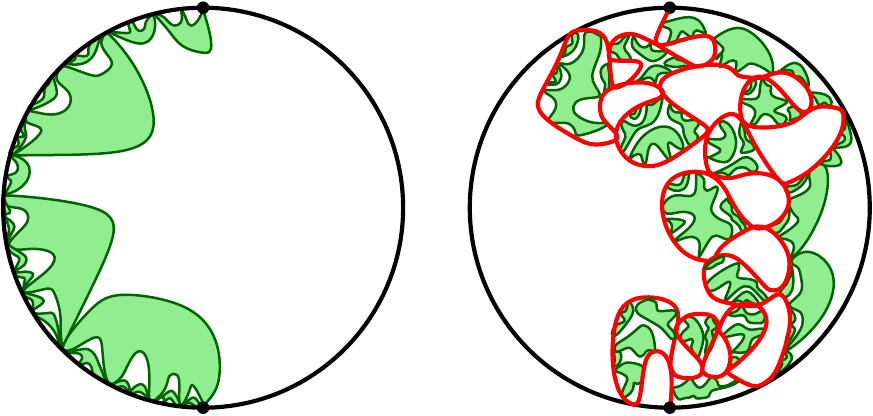}	
\end{center}
\caption{\label{fig:sle_bcle} {\bf Left:} A collection of loops in $\D$ created by a branching $\SLE_\kappa(3\kappa/2-6)$ process from $-i$ to $i$.  The right boundary of the loops corresponds to this process targeted at $i$.  The loops altogether form a $\BCLE_\kappa(-\kappa/2)$ with marked points $-i$ and $i$. {\bf Right:} An $\SLE_{\kappa'}(\kappa'-6)$ process $\eta'$  from in $\D$ from $-i$ to $i$.  In each clockwise loop or component to the right of $\eta'$ is an independent $\BCLE_\kappa(-\kappa/2)$ (their interior is filled).  Following the $\BCLE_\kappa(-\kappa/2)$ loops in the order they are first visited by $\eta'$ yields an $\SLE_\kappa^1(\kappa-6)$.}
\end{figure}

Note that the components of $D \setminus \eta'$ which are to the right of $\eta'$ have boundary which can be decomposed into two parts: the part which is in $\partial D$ and the part which is in the range of $\eta'$.  So, in the second step, the loops of $\eta$ which are in these components are drawn by independent branching $\SLE_\kappa(3\kappa/2-6)$ processes in each component, starting from the first point on the boundary which is visited by $\eta'$.  Moreover, the right boundary of these loops is given in each component by an $\SLE_\kappa(3\kappa/2-6)$ starting from the first point on the component boundary visited by $\eta'$ and targeted at the last.  It will be convenient in some places below to think of the concatenation of these $\SLE_\kappa(3\kappa/2-6)$ processes as a single curve (which we will do)  and refer to as an $\SLE_\kappa(3\kappa/2-6)$.

We can use the same trunk and the same collection of $\SLE_\kappa$-type loops to build a continuous path $\wt{\eta}$ from $y$ to $x$.  By reversibility, $\wt{\eta}$ has the law of an $\SLE_\kappa^{-1}(\kappa -6)$ from $y$ to $x$ that is tracing exactly the same loops at $\eta$ (but clockwise instead of counterclockwise).  Note that while the ranges of $\eta$ and of $\wt \eta$ coincide and the trunk of $\wt {\eta}$ is the time-reversal of $\eta'$, the two processes $\eta$ and $\wt {\eta}$ are not the time-reversal of each other.  Indeed, the trunk will meet each loop more than once, so that the order in which $\wt \eta$ encounters the loops is not exactly the reversed order in which $\eta$ meets them.  In a certain sense, this construction in fact provides a very precise description of the lack of reversibility of the $\SLE_\kappa(\kappa -6)$ processes.

\section{Conformal invariance of $\CLE_\kappa$ exploration tree hookup probabilities}
\label{sec:hookup}

The goal of the present section will be to derive a conformal invariance statement related to pairs of explorations of $\CLE_\kappa$'s.  In Section~\ref{subsec:hookup4_to_8}, we  will 
address the case that $\kappa \in (4,8)$, which is the one that will be an essential ingredient in the proofs of our main three theorems.  
For completeness and future reference, we also discuss the case that $\kappa \in (8/3,4)$ in Section~\ref{subsec:hookup_83_to_4}
(note that the story in the case that $\kappa=4$ is anyway much simpler because of the connection between the GFF and $\SLE_4(\rho)$ processes).

\subsection{The case $\kappa \in (4,8)$}
\label{subsec:hookup4_to_8}

\begin{figure}[ht!]
\includegraphics[scale=0.85,page=8]{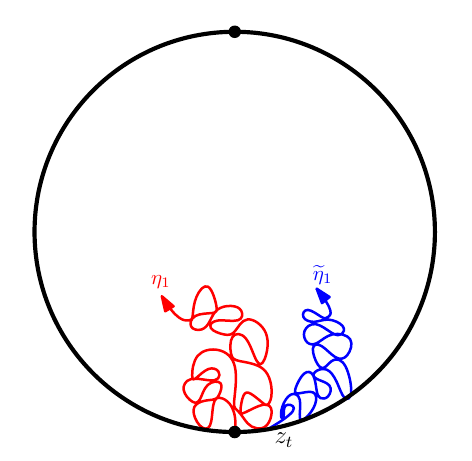}	
\includegraphics[scale=0.85,page=9]{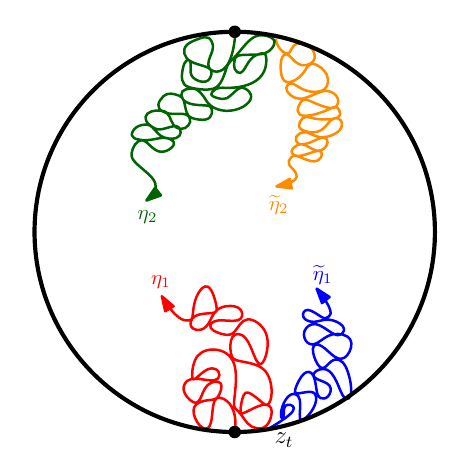}	
\caption{\label{fig:setup} Exploration of a $\CLE_\kappa$ for $\kappa \in (4,8)$ starting from $-i$ and from $i$, creating four branches.}
\end{figure}

Let us consider a $\CLE_\kappa$ for $\kappa \in (4, 8)$ in $\D$. Some of the $\CLE$ loops will hit $\partial \D$, and some others will not. For each loop $\CL$ that intersects the counterclockwise half-circle from $-i$ to $i$, we can define its first and last intersection points $z(\CL)$ and $\wt z (\CL)$ on this half-circle, when one moves from $-i$ to $i$. One can then define a continuous path $\eta_1^\#$ from $-i$ to $i$ as follows: One moves counterclockwise along the half-circle, and each time one hits $z(\CL)$ for some loop $\CL$, one attaches the clockwise loop $\CL$ to the path, and then proceeds.  We note that this defines a continuous path by the local finiteness of $\CLE_\kappa$ proved in \cite{MS_IMAG4}. One can then define the subpath $\eta_1$ of this path that corresponds to its growth as seen from $i$ (loosely speaking, one cuts out all parts that are growing while hidden away from~$i$). This path $\eta_1$ is the concatenation of all clockwise portions of loops $\CL$ from $z(\CL)$ to $\wt z(\CL)$ that are not disconnected from $i$ and $-i$ by any other such portion. The law of the path $\eta_1$ is that of an $\SLE_\kappa (\kappa -6)$ from $-i$ to $i$ in $\D$.  In fact, $\eta_1$ is exactly the branch from $-i$ to $i$ of the $\CLE_\kappa$ exploration tree.

One can note that after some time $t$ ($t$ can be a deterministic time or a stopping time with respect to the filtration generated by $\eta_1$), one can define $z(t)$ to be the last point on the half-circle from $-i$ to $i$ that $\eta_1$ visited before $t$. If $\eta_1(t) \not= z(t)$, then this point $z(t)$ is (by construction) equal to $z(\CL)$ where $\CL$ is the loop that $\eta_1$ is (partially) tracing at time $t$.

One can also interchange the roles of $i$ and $-i$ and perform the backward procedure: one moves clockwise along the half-circle from $i$ to $-i$, and attaches the loops of the $\CLE_\kappa$ drawn in counterclockwise manner. In this way, one defines a path $\eta_2^\#$ and a subpath $\eta_2$ (which is $\eta_2^\#$ seen as growing from $-i$). While the time-reversal of $\eta_1^\#$ is not identical to $\eta_2^\#$ because the order of the loops and the way in which they are discovered have been changed, the time-reversal of $\eta_1$ is exactly $\eta_2$ (it is described also via the concatenation of the same portions of loops $\CL$ between  $\wt z (\CL)$ and $z (\CL)$):  The time-reversal of the  $\SLE_\kappa (\kappa -6)$ $\eta_1$ from $-i$ to $i$ is the $\SLE_\kappa (\kappa -6)$ $\eta_2$ from $i$ to $-i$  (modulo the convention that the marked point is now on the other side of the path).

Let us now suppose that one has discovered $\eta_1$ up to a stopping time $t$ and that $\eta_1 (t) \not= z(t)$. As we have already mentioned, the point $z(t)$ is then the starting point of the $\CLE$ loop $\CL$ that $\eta_1 (t)$ is part of. In particular, the conditional law of the rest of this loop given $\eta_1|_{[0,t]}$ is exactly an $\SLE_\kappa$ from $\eta_1(t)$ to $z(t)$ in the component of $\D \setminus \eta_1 ([0,t])$ with $z(t)$ on its boundary. We can now decide to discover part of this loop counterclockwise starting from $z(t)$. By {the} time-reversal {symmetry} of $\SLE_\kappa$ \cite{MS_IMAG3}, the law of this path is an $\SLE_\kappa$ from $z(t)$ to $\eta_1(t)$ in the component of $\D \setminus \eta_1 ([0,t])$ with $\eta_1(t)$ on its boundary. Let us call this path $\wt \eta_1$, and discover $\wt \eta_1$ up to some stopping time $s$ (note that the definition of $\wt \eta_1$ and the notion of stopping time depend on $\eta_1|_{[0,t]}$). Given $\eta_1$ up to time $t$ and $\wt \eta_1 $ up to time $s$,
 the law of the missing part of $\CL$ joining~$\eta_1 (t)$ to $\wt \eta_1 (s)$ is just an $\SLE_\kappa$ in the remaining to be discovered domain (i.e.\ in  the connected component $D_{t,s}$ of $\D \setminus (\eta_1 ([0,t]) \cup \wt \eta_1 ([0,s]))$ which has $\wt \eta_1 (s)$ and $\eta_1 (t)$ on its boundary)
because of the reversibility of $\SLE$.

We now define symmetrically the path~$\eta_2$ up to some stopping time~$t_2$, the point $\wt z (t_2)$ and the path~$\wt \eta_2$ from $\wt z (t_2)$ to $\eta_2 (t_2)$. We assume that we are in a configuration as depicted in Figure~\ref{fig:setup}, where all four points $\eta_1 (t)$, $\wt \eta_1 (s)$, $\eta_2 (t_2)$ and $\wt \eta_2 (s_2)$ are four different boundary points of the same connected component $D(t,s,t_2,s_2)$ of the remaining to be discovered domain. Typical examples of stopping times $t$, $s$, $t_2$ and $s_2$ can be the respective hitting times of a circle of radius $r$ around the origin by the respective four strands (if they do make it to that circle).

We can note that conditionally on these four branches, two possibilities arise: 
\begin{itemize}
 \item $\eta_1(t)$ and $\eta_2 (t_2)$ correspond to parts of the same $\CLE$ loop. In this case,
 $\wt \eta_1$ will hook up with $\wt \eta_2$, while the path $\eta_1$ will first hook up with $\eta_2$ (and these last two paths respectively coincide with $\eta_1^\#$ and $\eta_2^\#$ up to when they hook up). We call this the one-loop event $E_1$.
 \item $\eta_1 (t)$ and $\eta_2 (t_2)$ correspond to different $\CLE$ loops. In this case, $\wt \eta_1$ will hook up with $\eta_1^\#$ without meeting $\wt \eta_2 (s_2)$, and $\wt \eta_2$ will hook up with $\eta_2^\#$ without meeting $\wt \eta_1 (s)$. We call this the two-loop event $E_2$.
\end{itemize}

With the previous notation, we define $C(t,s,t_2, s_2)$ to be the configuration given by the domain $D(t, s, t_2, s_2)$ and the four counterclockwise ordered boundary points 
$\eta_1 (t), \wt \eta_1 (s),  \wt \eta_2 (s_2), \eta_2 (t_2)$.
We are going to establish the following conformal invariance statement. 

\begin{lemma}
\label{lem:paths_confinv}
If we are in a configuration as depicted in Figure~\ref{fig:setup}, the conditional distribution of the remaining pieces of $\eta_1$, $\wt \eta_1$, $\eta_2$ and $\wt \eta_2$ in $D(t,s,t_2, s_2)$ until they hook up into one or two loops is a conformally invariant function of the configuration $C(t,s,t_2, s_2)$.
\end{lemma}

In fact, in order to prove this lemma it suffices to prove the following seemingly weaker result:  
\begin{lemma}
 \label{lem:paths_hookup}
If we are in a configuration as depicted in Figure~\ref{fig:setup}, the conditional distribution of $E_1$ (and therefore of $E_2$) is a function $f_\kappa (\cdot)$ of the cross-ratio between the four boundary points in the configuration $C(t,s,t_2, s_2)$.
\end{lemma}
Indeed, conditionally on $E_1$, we can describe the joint conditional law of the remaining pieces of the loops containing $\eta_1$ and $\wt \eta_1$ (and therefore of $\eta_2$ and $\wt \eta_2$) in $D(t,s,t_2, s_2)$ as the bi-chordal $\SLE_\kappa$ joining the four end-points in that domain, which is characterized uniquely by the fact that conditionally on one of the two paths, the law of the other one 
is an ordinary $\SLE_\kappa$ in the remaining domain (see \cite[Theorem 4.1]{MS_IMAG2}), and we know that this property is satisfied in the present case (the same argument can also be applied when one conditions on $E_2$). 

In fact, we can notice that this conditional distribution is in fact symmetric when one formally interchanges the roles of $(\eta_1 (t), \wt \eta_1 (s), \wt \eta_2 (s_2), \eta_2 (t_2))$ 
and $( \wt \eta_2 (s_2), \eta_2 (t_2) , \eta_1 (t), \wt \eta_1 (s))$. It therefore follows that it is in fact sufficient to prove Lemma~\ref{lem:paths_hookup} in the case where we do not grow $\wt \eta_1$ or $\wt \eta_2$ after $\eta_1$ and $\eta_2$ have been defined (in other words, when $s_2=t_2=0$). Indeed, we can then decide to switch the roles of $\eta_1, \eta_2$ and $\wt \eta_1, \wt \eta_2$, and to continue growing $\eta_1$ and $\eta_2$ instead of growing $\wt \eta_1$ and $\wt \eta_2$, and by conformal invariance, the general case of Lemma~\ref{lem:paths_hookup} follows.

As we have already pointed out,  $\eta_1$ (and its time-reversal $\eta_2$) is an $\SLE_\kappa (\kappa-6)$ process, so that the following result implies Lemma~\ref{lem:paths_hookup} (here, modulo conformal invariance, $\eta$ plays the role of the remaining part of $\eta_1$, including $\eta_2$, while $\eta_2$ plays the role of the beginning of the time-reversal of $\eta$).

\begin{lemma}
\label{lem:sle_k_k_minus_6_reversal}
Fix $\kappa \in (4,8)$ and suppose that~$\eta$ is an $\SLE_\kappa(\kappa-6)$ process in~$\h$ from~$0$ to~$\infty$ with force point located in~$\R_+$.  Let~$\eta_R$ be the time-reversal of~$\eta$, let~$\tau_R$ be an~$\eta_R$-stopping time, and let $D_{\tau_R}$ be the component of~$\h \setminus \eta_R([0,\tau_R])$ with~$0$ on its boundary.  Then the conditional law of~$\eta$ given $\eta_R|_{[0,\tau_R]}$ viewed as a path in $D_{\tau_R}$ is a conformally invariant function of the configuration consisting of the domain~$D_{\tau_R}$ and the four marked boundary points given by~$0$, the location of the force point, $\eta_R(\tau_R)$, and  $\min ( \eta_R|_{[0,\tau_R]} \cap \R_+)$.
\end{lemma}

\begin{proof}
Note that the case $\kappa =6$ is trivial. We will treat the two cases $\kappa \in (6,8)$ and $\kappa \in (4,6)$ separately.  The distinction between these two cases reflects the change in sign of $\kappa -6$.  When $\kappa \in (6,8)$, we will use the $\SLE$/GFF coupling \cite{MS_IMAG} and the reversibility of $\SLE_\kappa$ \cite{MS_IMAG3}.  When $\kappa \in (4,6)$ we will use the $\CLE$ setup and results from \cite{cle_percolations} as reviewed just above in Section~\ref{subsec:percolation_review}.

The proof in both cases will make use of a variant of the resampling characterization of \emph{bi-chordal $\SLE$}.  More precisely, we will use a more general version of \cite[Theorem~4.1]{MS_IMAG2}, {which is proved in Appendix~\ref{sec:bichordal}}, that states that there is a unique law on pairs of curves $(\eta_1,\eta_2)$ connecting a pair of boundary points $x,y$ with~$\eta_1$ to the left of~$\eta_2$ so that the conditional law of~$\eta_1$ given~$\eta_2$ (resp.\ $\eta_2$ given $\eta_1$) is that of a certain $\SLE_\kappa(\rho)$ type process with force points on its left (resp.\ right) side.

\begin{figure}
\includegraphics[scale=0.85, page=16]{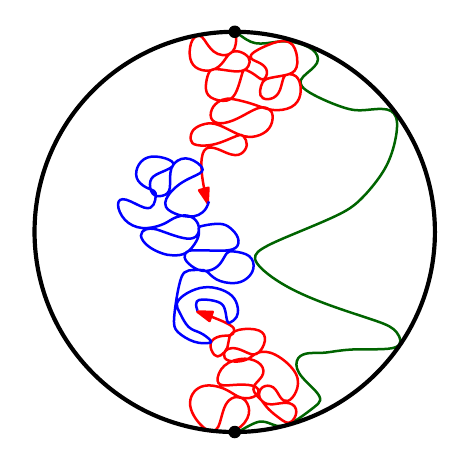}
\caption{\label{fig:sle_k_k_minus_6_reversal_6_to_8} Illustration of the setup used in the proof of Lemma~\ref{lem:sle_k_k_minus_6_reversal} in the case that $\kappa' \in (6,8)$.  The red paths show an $\SLE_{\kappa'}(\kappa'-6)$ path $\eta'$ and its time-reversal $\eta_R'$ from $-i$ to $i$ in $\D$ coupled with a GFF instance as a counterflow line drawn up to forward and reverse stopping times, $\tau$ and $\tau_R$.  The blue path is the remainder of $\eta'$.  Shown in green is the flow line $\eta$ of this GFF from $-i$ to $i$ with angle $\theta = 3\pi/2 - 2\lambda/\chi$.  It follows from \cite{MS_IMAG} that the conditional law of $\eta'$ given $\eta$ is that of an $\SLE_{\kappa'}$ process in the component of $\D \setminus \eta$ which is to the left of $\eta$.  By the reversibility of $\SLE_{\kappa'}$ proved in \cite{MS_IMAG}, the conditional law of $\eta'$ given $\eta'|_{[0,\tau]}$, $\eta_R'|_{[0,\tau_R]}$, and $\eta$ is that of an $\SLE_{\kappa'}$ process in the remaining domain.  Conversely, the 
conditional law of $\eta$ given all of $\eta'$ is independently that of an $\SLE_\kappa(\kappa-4;2-3\kappa/2)$ in the components which are to the right of $\eta'$.  Thus as these conditional laws are conformally invariant, the conformal invariance of the joint law of $\eta'$ and $\eta$ given $\eta'|_{[0,\tau]}$ and $\eta_R'|_{[0,\tau_R]}$ follows from the bi-chordal arguments of \cite[Section~4]{MS_IMAG2}.}	
\end{figure}

\begin{figure}
\includegraphics[scale=0.85, page=17]{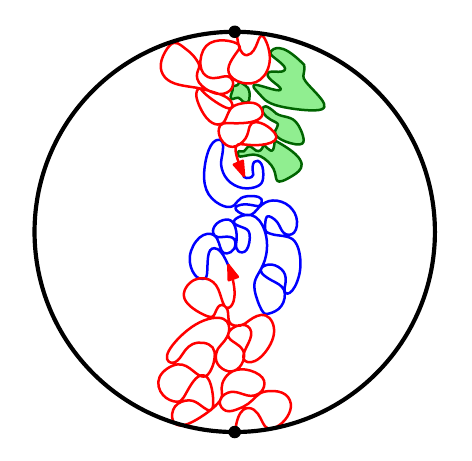}
\caption{\label{fig:sle_k_k_minus_6_reversal_4_to_6} Illustration of the setup used in the proof of Lemma~\ref{lem:sle_k_k_minus_6_reversal} in the case that $\kappa' \in (4,6)$. Shown is an $\SLE_{\kappa'}(\kappa'-6)$ process from $-i$ to $i$ in $\D$ viewed as a CPI in a $\CLE_\kappa$, $\kappa=16/\kappa' \in (8/3,4)$, process $\Gamma$.  If we condition on $\eta'$ up to a stopping time $\tau$, the time-reversal $\eta_R'$ of $\eta'$ up to an $\eta_R'$-stopping time $\tau_R$ and the loops of $\Gamma$ which touch this path segment (green loops), then the conditional law of the rest of $\eta'$ (blue path) is an $\SLE_{\kappa'}(\kappa'-6)$ in the remaining domain.  Conversely, if we condition on all of $\eta'$ (red and blue paths), then the conditional law of the loops which touch $\eta'$ is given by independent $\ccwBCLE_\kappa(-\kappa/2)$'s in the components of $\D \setminus \eta'$ which are surrounded by the right side of $\eta'$.  Thus as these conditional laws are conformally invariant, the conformal 
invariance of the joint law follows from the bi-chordal arguments of {Appendix~\ref{subsec:bc_three_paths}}.}
\end{figure}

We choose again the following notation: we fix $\kappa' \in (4,8)$ and let $\kappa = 16/\kappa' \in (2,4)$.  We first consider the case that $\kappa' \in (6,8)$; see Figure~\ref{fig:sle_k_k_minus_6_reversal_6_to_8} for an illustration of the argument.  (The reader may find it helpful to look at \cite[Figure~2.5]{MW_INTERSECTIONS}.)  We can view $\eta'$ as the counterflow line from $\infty$ to $0$ of a GFF $h$ on $\h$ with constant boundary conditions given by $-\lambda' + \pi \chi$ on $\R$.  Let $\eta$ be the flow line of $h$ with angle $\theta = 3\pi/2 - 2\lambda/\chi$ from $0$ to $\infty$.  Then $\eta$ is an $\SLE_\kappa(3\kappa/2-4;2-3\kappa/2)$ process.  (Note that $2-3 \kappa/2 > -2$ provided $\kappa < 8/3$ and $\kappa' > 6$.)  Let $\eta_R'$ be the time-reversal of $\eta'$ and let $\tau,\tau_R$ be stopping times for $\eta',\eta_R'$, respectively.  Then:
\begin{itemize}
\item It follows from \cite{MS_IMAG} that the conditional law of $\eta$ given $\eta'$ is independently that of an $\SLE_\kappa(\kappa-4;2-3\kappa/2)$ process in each of the components of $\h \setminus \eta'$ which are to the right of $\eta'$ and whose boundary have non-empty intersection with $\partial \h$.
\item The conditional law of $\eta'$ given $\eta$ is that of an $\SLE_{\kappa'}$ process in the component of $\h \setminus \eta$ which is to the left of $\eta$.  Consequently, by the reversibility of $\SLE_{\kappa'}$ processes for $\kappa' \in (4,8)$ proved in \cite{MS_IMAG3}, it follows that the conditional law of $\eta'$ given $\eta$, $\eta'|_{[0,\tau]}$, and $\eta_R'|_{[0,\tau_R]}$ is that of an $\SLE_{\kappa'}$ process in the remaining domain.
\end{itemize}
Since the two conditional laws are conformally invariant given $\eta'|_{[0,\tau]}$ and $\eta_R'|_{[0,\tau_R]}$, it follows from the bi-chordal $\SLE$ characterization ({Appendix~\ref{subsec:bc_two_paths}}) that the joint law of $\eta$ and $\eta'$ given $\eta'|_{[0,\tau]}$ and $\eta_R'|_{[0,\tau_R']}$ is conformally invariant.

We next consider the case that $\kappa' \in (4,6)$; see Figure~\ref{fig:sle_k_k_minus_6_reversal_4_to_6} for an illustration of the argument.  We let $\eta'$ be an $\SLE_{\kappa'}(\kappa'-6)$ process in~$\h$ from~$0$ to~$\infty$ {with force point located at $0^+$}.  Let $\eta_R'$ be the time-reversal of $\eta'$ and let~$\tau$ (resp.\ $\tau_R$) be a stopping time for~$\eta'$ (resp.\ $\eta_R'$).  We view $\eta'$ as a CPI (in the sense of \cite[Definition~2.1]{cle_percolations}) coupled with a~$\CLE_\kappa$, say~$\Gamma$, in~$\h$.  We note that then $\eta_R'$ is also a CPI coupled with $\Gamma$.  The CPI property of~$\eta_R'$ implies that $\eta_R'|_{[\tau_R,\infty)}$ is a CPI associated with the~$\CLE_\kappa$ given by including those loops of~$\Gamma$ which are contained in the complementary component of $\eta_R'([0,\tau_R])$ with $0$ on its boundary.  In particular, conditioned on this we have that the law of the remainder of $\eta'$ given $\eta|_{[0,\tau]}$ and $\eta_R'|_{[0,\tau_R]}$ and the loops of $\Gamma$ which hit $\eta_R'|_{[0,\tau_R]}$ is that of an $\SLE_{\kappa'}(\kappa'-6)$ in the remaining domain.  The same also holds if we switch the roles of $\eta'$ and $\eta_R'$.

Summarizing, we have that:
\begin{itemize}
\item Given $\eta'|_{[0,\tau]}$, $\eta_{R}'|_{[0,\tau_R]}$, and the loops of~$\Gamma$ which hit $\eta_R'|_{[0,\tau_R]}$, the remainder of $\eta'$ has the law of an $\SLE_{\kappa'}(\kappa'-6)$ process in the remaining domain.
\item {Given $\eta'|_{[0,\tau]}$, $\eta_{R}'|_{[0,\tau_R]}$, and the loops of~$\Gamma$ which hit $\eta'|_{[0,\tau]}$, the remainder of $\eta_R'$ has the law of an $\SLE_{\kappa'}(\kappa'-6)$ process in the remaining domain.}
\item Given all of $\eta'$, the conditional law of the loops of $\Gamma$ which hit $\eta'$ is given by a $\ccwBCLE_\kappa(-\kappa/2)$ in each of the complementary domains which are to the right of $\eta'$.
\end{itemize}
As explained in Appendix~\ref{subsec:bc_three_paths}, the conformal invariance of these three laws implies that the joint law of $\eta'$ and the loops of $\Gamma$ which hit $\eta'|_{[0,\tau]}$ and $\eta_R'|_{[0,\tau_R]}$ is conformally invariant given $\eta'|_{[0,\tau]}$ and $\eta_R'|_{[0,\tau_R]}$ in the remaining domain.
\end{proof}

As explained in \cite{mw2017connection}, building on this lemma, on Dub\'edat's commutation relations \cite{dub2007commutation} and some SLE estimates, it is actually possible to explicitly identify the hookup probability function $f_\kappa$ in terms of a ratio of hypergeometric functions.  We conclude this subsection with the following simpler result that just states that the function $f_\kappa$ is well-behaved (this will be sufficient for the purpose of the present paper).  Here and in the sequel, we refer to the definition of the cross-ratio of a conformal rectangle to be defined on $(0,\infty)$ and equal to $1$ for a conformal square such as the unit disk with the four boundary points $1, i , -1, -i$:

\begin{lemma}
\label{lem:paths_hookup2}
The function $f_\kappa (\cdot)$ is bounded away from $0$ and from $1$ on any compact subset of $(0, \infty)$. 
The function $f_\kappa(c)$ converges to either $0$ or $1$ as $c \to 0$ or $c\to \infty$. 
\end{lemma}
\begin{proof}
It for instance suffices to start from the configuration $C(t, s, t_2, s_2)$ with cross-ratio $c$ and to note that it is possible (with positive probability $p_\kappa (c)$) to let $\eta_1$ grow further until the new cross-ratio hits $1$. Hence, we get that $f_\kappa ( c) \ge  p_\kappa(c) f_\kappa(1)$. The same argument can be applied to the two-loop event.

The result about the limit of $f_\kappa (c)$ as $c \to 0$ and $c \to \infty$ follows from the fact that when one explores one strand (say $\eta_1$) until then end,
one will eventually discover which hookup event holds. So, just before that moment, the conditional probability of the hookup event will tend to $0$ or $1$. But by construction, the cross-ratio will tend to $0$ or $\infty$ at the same time. 
\end{proof}

\subsection{The case $\kappa \in (8/3,4)$}
\label{subsec:hookup_83_to_4}

We are now going to establish the analog of Lemma~\ref{lem:paths_hookup} for the case $\kappa \in (8/3,4)$.  The setup will take a slightly different form than in the case of Lemma~\ref{lem:paths_hookup} because we cannot use the fact that the branches of the exploration tree used to build a $\CLE_\kappa$ are deterministic functions of the $\CLE_\kappa$ (indeed, it is one of the main results of the present paper that it is not the case), so we need to first explain how we define the joint law of the two explorations.  As we mentioned earlier, the content of the present subsection will not be used later in the present paper and it is included here for future reference (it is used in \cite{mw2017connection}).  We also leave out the (easier) case $\kappa =4$, as this one can be dealt with via the relation between $\CLE_4$ and the Gaussian free field (see for instance \cite {mw2017connection}).

Throughout, we suppose that we have a simply connected domain $D \subseteq \C$, that $x$ and $y$ are distinct boundary points, and that $\eta$ is an $\SLE_\kappa^1(\kappa-6)$ path in~$D$ from~$x$ to~$y$.  We also suppose that~$\wt{\eta}$ is the $\SLE_\kappa^{-1}(\kappa-6)$ from~$y$ to~$x$ whose trunk and loops are the same as those of~$\eta$, as described in Section~\ref{subsec:trunk_construction}.  We emphasize again that~$\wt{\eta}$ is not exactly the time-reversal of~$\eta$, but that the trace of $\eta$ and $\wt \eta$ coincide, that the trunk of $\wt{\eta}$ is the time-reversal of that of $\eta$, and that $\wt{\eta}$ visits exactly the same $\SLE_\kappa$-type loops as $\eta$.

When we explore $\eta$, we can choose any deterministic way to parameterize it (so that its ``time'' is a deterministic function of its trace). We will then use the natural filtration $({\mathcal F}_t)$ generated by the path. Note that (because the $\SLE_\kappa$ loops are simple while the trunk is not a simple curve), at a stopping time $\tau$, the knowledge of the last point on the trunk that has been visited by $\eta$ before time $\tau$ is contained in the $\sigma$-field ${\mathcal F}_\tau$.  We call this point $X(\tau)$, and we use the same notation for $\wt \eta$ (thus defining $\wt X (\wt \tau)$).

We are now going to discover a piece of $\eta$ and a piece of $\wt \eta$. More precisely, suppose that $\tau$ (resp.\ $\wt{\tau}$) is a stopping time for $\eta$ (resp.\ $\wt{\eta}$). 
We define the event 
\[{A}(\tau, \wt \tau) := \{ \eta([0,\tau]) \cap \wt{\eta}([0,\wt{\tau}]) = \emptyset, \ \eta (\tau) \not= X( \tau), \ \wt \eta (\wt \tau) \not= \wt X (\wt \tau)\}.\]
On this event, we call $D_{\tau, \wt \tau}$ the connected component of the complement of $\eta([0,\tau]) \cup \wt \eta([0, \wt \tau])$ that has $\eta (\tau)$ on its boundary. Note that $X(\tau)$ and $\wt X (\wt \tau)$ each correspond to two prime ends in $D_{\tau, \wt \tau}$. We will consider implicitly the one that is on the 
left-hand side of $\eta (\tau)$ (resp.\ the right-hand side of $\wt \eta (\wt \tau)$), i.e., on the side of the trunk.

\begin{lemma}
\label{lem:cle_k_83_4_conf_inv}
The conditional probability given ${\mathcal F}_\tau$,  $\wt {\mathcal F}_{\wt \tau}$ and the event $A = A(\tau, \wt \tau)$, of the event that $\eta(\tau)$ and $\wt \eta (\wt \tau)$ are part of the same $\CLE_\kappa$ loop is a function of the cross-ratio of the four marked points $\eta(\tau)$, $\wt{\eta}(\wt{\tau})$, $X(\tau)$ and $\wt X (\wt \tau)$ in $D_{\tau,\wt \tau}$.
\end{lemma}

\begin{proof} 
Let~$\eta'$  denote the trunk of~$\eta$ up until first hitting $X=X (\tau)$. Define $\wt{\eta}'$ to be the trunk of $\wt{\eta}$  up until hitting $\wt X = \wt X (\wt \tau)$.  We let $\wh \eta'$ be the missing middle piece of the trunk (joining $X$ and $\wt X$), so that the concatenation of $\eta'$, $\wh{\eta}'$, and (the time-reversal of) $\wt{\eta}'$ together form the trunk of $\eta$.

Let us define $E:= (\eta(t), t \le \tau)$ (resp.\ $\wt{E} = (\wt {\eta} (t), t \le \wt \tau)$), so that the $\sigma$-field generated by $E$ (resp.\ $\wt{E}$) is ${\mathcal F}_\tau$ (resp.\ $\wt{\CF}_{\wt{\tau}}$).

Let us call~$L$ the remaining part of the loop that $\eta$ has started to trace at time $\tau$, and let similarly~$\wt{L}$ be the remaining part of the loop of $\wt {\eta}$ that $\wt {\eta}$ has started to trace at time $\wt {\tau}$. Note that $L$ could contain $\wt { \eta} ( \wt {\tau})$ if $\eta$, $\wt{\eta}$, respectively, are exploring the same loop at time $\tau$, $\wt{\tau}$.

Using resampling techniques, we will show that the conditional law of $(L, \wt L)$ is a conformally invariant function of the domain and the four marked points (which clearly implies the lemma).  

We first note that conditionally on $(E , L)$, the law of $\wt {\eta}$ up until the time at which it hits $\eta([0,\tau]) \cup L$ is that of the beginning of an $\SLE_\kappa^{-1} (\kappa -6)$ from $y$ to $X$ in the component of the complement of $\eta([0,\tau]) \cup L$ with $y$ on its boundary.  If we now condition on $F:= (E ,L , \wt {E} , \wt {L})$ (and suppose that the event $A$ holds), and let $U$ denote the connected component of the complement of $\eta([0,\tau]) \cup L \cup \wt {\eta}([0 ,\wt {\tau}]) \cup \wt {L}$ that lies outside of $L$ and $\wt{L}$, and has $X$ and $\wt{X}$ on its boundary (which is the one where the middle piece of the trunk will be), we see that this middle piece of the trunk (i.e., $\wh {\eta}'$) will join $X$ and $\wt X$ in this domain. In fact, its conditional distribution (given $F$) is that of an $\SLE_{\kappa'}(\kappa' - 6)$ from $X$ to $\wt X$ in $U$ (this follows directly from  the conformal Markov property of the exploration mechanism: One can first discover $(E,  L)$, 
then discover $\wt \eta$ up to the time at which it hits $L$ or finishes drawing $\wt L$, and see that $\wh{\eta}'$ is the trunk of an $\SLE_\kappa^{1} (\kappa-6)$ in the remaining domain). 

Conversely, if we condition on the entire trunk, then we can describe the law of all the $\CLE_\kappa$ loops that are attached to it.  In particular, the outer boundary of the set of CLE loops of $\Gamma$ that touch the trunk is distributed like an $\SLE_\kappa(3\kappa/2-6)$ from $x$ to $y$ in the collection of complementary components which lie ``to the right'' of the trunk (in the sense described at the end of Section~\ref{subsec:trunk_construction} -- note that there could be in fact a countable collection of such components).  We can note that both $\eta (\tau)$ and $\wt \eta (\wt \tau)$ will be on this path. 

In particular, if we further condition on $(E ,  \wt {E})$, we get an $\SLE_\kappa (3 \kappa /2 - 6)$ conditioned on part of its beginning and part of its end. Recall from \cite[Theorem~6.2]{MS_IMAG2} that this conditional distribution is conformally invariant.  Once this right boundary is completed, one can easily complete the picture in a conformally invariant way, using the procedure described in Section~\ref{subsec:trunk_construction}. 

In the present setup, this implies in particular that the conditional distribution of $(L,\wt {L})$ has the following properties: 
We consider the connected components of the complement of  $\eta([0,\tau]) \cup \wt{\eta}([0 ,\wt{\tau}]) \cup \wh{\eta}'$ that have  $\eta(\tau)$ or $\wt{\eta}(\wt{\tau})$ on their boundary. There are two possibilities here: Either there is just one such component (that has both points on their boundary) in which case both hookup scenarios are possible, or there are two components (and then $L$ and $\wt {L}$ must be in these two different components). Then, the conditional distribution of $(L , \wt {L})$ is conformally invariant 
in the sense that: 
\begin {itemize} 
 \item In the former case (with one connected component), it is a conformally invariant function of the connected component 
 with the four marked points $\eta (\tau), \wt \eta (\wt {\tau}), X, \wt {X}$.
 \item In the latter case (with two connected components), $L$ and $\wt {L}$ are conditionally independent, and the conditional law of $L$ (resp.\ $\wt L$) is a conformally 
invariant function of its corresponding component and the two marked boundary points $\eta(\tau)$ and $X$ (resp. $\wt \eta ( \wt \tau)$ and $\wt X$).
\end {itemize}
So, the previous two paragraphs establish the existence and the conformal invariance of: 
\begin {itemize}
 \item The conditional distribution of $(L, \wt {L})$ given $(E , \wt {E} , \wh{\eta}')$. 
 \item The conditional distribution of $\wh{\eta}'$ given $(L, \wt {L},E, \wt {E})$. 
 \end {itemize}
We will make use of the conditional distribution of $L$ given $\wt{L}$ (and $E$, $\wt{E}$) as well as the conditional distribution of $\wt{L}$ given $L$ (and $E$, $\wt{E}$).  As we explained earlier, we note that if we condition on $\wt{E}$ and $\wt{L}$, then the conditional law of $\eta$ in the remaining domain (up until it hits $\wt{\eta}([0,\wt{\tau}])$) is that of an $\SLE_\kappa(\kappa-6)$ process.  On the event that $L$ and $\wt{L}$ are distinct, it thus follows that the conditional law of $L$ given $E$, $\wt{E}$, and $\wt{L}$ is that of an $\SLE_\kappa$ process from $\eta(\tau)$ to $X(\tau)$ in the component of $D \setminus ( \eta([0,\tau]) \cup \wt{\eta}([0,\wt{\tau}]) \cup \wt{L})$ with $\eta(\tau)$ on its boundary.  The same reasoning implies that we conversely have that the conditional law of $\wt{L}$ given $L$, $E$, and $\wt{E}$ is that of an $\SLE_\kappa$ process from $\wt{\eta}(\wt{\tau})$ to $\wt{X}(\wt{\tau})$ in the component of $D \setminus (\eta([0,\tau]) \cup L \cup \wt{\eta}([0,\wt{\tau}]
))$ with $\wt{\eta}(\wt{\tau})$ on its boundary.

We are in a setup where we can apply the resampling ideas of the type described in the appendix. We want to show that these two conditional distributions characterize uniquely the conditional law of $(L, \wt {L}, \wh{\eta}')$ given $(E, \wt {E})$ (and this will in particular imply the conformal invariance of this conditional distribution and the conformal invariance of the hookup probability).  We will consider the resampling kernel which is defined from the above resampling kernels as follows:
\begin{itemize}
\item With probability $1/3$, we resample $L$ given $(\wt{L},E,\wt{E})$ from its conditional law and then resample $\wt{L}$ given the new realization of $L$ and $(E,\wt{E})$ from its conditional law.  We then resample $\wh{\eta}'$ from its conditional law given the new realization of $(L,\wt{L})$ and $(E,\wt{E})$.
\item With probability $1/3$, we resample $(L,\wt{L})$ from its conditional law given $(\wh{\eta}',E,\wt{E})$.
\item With probability $1/3$, we leave the configuration unchanged.
\end{itemize}

Suppose that $(L_1, \wt{L}_1, \wh{\eta}_1', E, \wt{E})$ and $(L_2, \wt{L}_2, \wh{\eta}_2', E, \wt{E})$ are sampled conditionally independently given $E$, $\wt{E}$ from two laws which are invariant under the aforementioned resampling kernel.  What is needed to apply the ergodicity-based argument is to show that we can apply this resampling kernel a finite number of times to yield two configurations which coincide with positive probability (we will show that it in fact suffices to apply this resampling kernel twice).

Let us first suppose that we are on the event that $(L_1, \wt {L}_1)$ and $(L_2, \wt {L}_2)$ correspond to configurations with two loops (i.e., the event that the loops being drawn by $\eta,\wt{\eta}$ at the times $\tau,\wt{\tau}$, respectively, are not the same).  (In the remaining paragraphs, $(E, \wt E)$ are considered to be fixed).  In this case, there is a positive chance that in the first application of the resampling kernel we resample $L_i$ then $\wt{L}_i$ for $i=1,2$ from its conditional law.  As we mentioned above, the conditional law of $L_i$ given $\wt{L}_i$ for $i=1,2$ is that of an $\SLE_\kappa$ process in the remaining domain and the same is likewise true for the conditional law of $\wt{L}_i$ given $L_i$.  Therefore the configurations $(L_1, \wt{L}_1)$ and $(L_2,\wt{L}_2)$ will be coupled on a common probability space to agree with positive probability.  On this event, the next part of this resampling step can be taken to resample $\wh{\eta}_i'$ for $i=1,2$ to coincide.  Therefore, in this 
case, we have coupled the two entire configurations to agree with positive probability after one application of the resampling kernel.

Now let us suppose that we are on the event that $(L_1, \wt {L}_1)$ and/or $(L_2, \wt {L}_2)$ correspond to configurations with only one loop (i.e., the event that the loops being drawn by $\eta,\wt{\eta}$ at the times $\tau,\wt{\tau}$ are the same).  In this case, there is a positive chance that in the first application of the resampling kernel we resample $(L_i, \wt {L}_i)$ given $\wh{\eta}_i'$ and then $L_i$ and $\wt{L}_i$ will not be part of the same loop (this follows easily from the fact that $\SLE_\kappa (3 \kappa /2 - 6)$ does hit the boundary and from the conformal invariance of $\SLE_\kappa (3 \kappa /2 - 6)$ conditioned by part of its beginning and part of its end, so that the conditional probability that after resampling one has a configuration with two loops, is a function of the cross-ratio of the four marked points). Therefore after this one iteration step, there is a positive chance that we obtain a configuration with two loops.  As explained in the previous paragraph, a second application of 
the resampling kernel will yield a configuration in which both configurations exactly coincide with positive probability.
\end {proof}

Just as in the case $\kappa \in (4,8)$, the CLE hookup probabilities for $\kappa \in (8/3,4)$ are actually worked out in the paper \cite{mw2017connection}, building among other things on the present Lemma~\ref{lem:cle_k_83_4_conf_inv}, on commutation relation considerations and on some SLE estimates.  

A final remark is that in Lemma \ref {lem:cle_k_83_4_conf_inv},
we could have in fact chosen $\wt {\tau}$ to be a stopping time 
for the filtration  $\sigma (E , (\wt \eta (s), s \le t))$ (i.e., one can use information about $E$ to choose the stopping time $\wt \tau$) without changing the proofs. This fact can turn out be handy, as one can for instance choose $\wt \tau$ to be the first time at which the cross-ratio between the four marked points reaches a certain value.

\section{Proof of Theorem~\ref{thm:cle_not_determined}}
\label{sec:proofs}

This section is devoted to the proof of  the fact that the $\CLE_\kappa$ loops are not determined by the $\CLE_\kappa$ gasket when $\kappa \in (4,8)$. We will explain in the subsequent section what modifications in the argument of the proof of this result enable us to also establish  Theorem~\ref{thm:path_not_determined} and Theorem~\ref{thm:percolation_not_determined}. 

Throughout this section, $\kappa \in (4,8)$ is fixed and all constants that will appear in the proofs can depend on our choice of $\kappa$.

\subsection{Notation}
\label{subsec:notation}

When $\ul z = (z_1, z_2, z_3, z_4)$ is a 4-tuple of counterclockwise ordered points on the unit circle, we denote by $\p_{\ul z}$ the joint law of the configuration with the four strands $\eta_1, \wt \eta_1, \wt \eta_2$ and~$\eta_2$  that were described in the previous section, starting respectively from these four points. More precisely, this is the conformal image of the law described in Lemma~\ref{lem:paths_confinv}.

From now on, we will actually denote these paths by $\gamma_1, \gamma_2, \gamma_3$ and $\gamma_4$, and we will also use the notation ${\ul \gamma} = ( \gamma_1, \gamma_2, \gamma_3, \gamma_4)$.
We note that  $\gamma_1$ (resp.\ $\gamma_3$) hooks up with either $\gamma_2$ or $\gamma_4$ (and then ends at $z_2$ or $z_4$) but does not hook up with $\gamma_3$ (resp.\ $\gamma_1$). 
Recall that the conditional law of $\gamma_3$ given all of $\gamma_1$ is that of an $\SLE_{\kappa}$ process from its initial to its target point in the complement of $\gamma_1$ and that the same is true when we switch the roles of $\gamma_1$ and $\gamma_3$.  This show that the paths do various things with positive probability because one can first sample how they are hooked up using Lemma~\ref{lem:paths_hookup} and then resample each of the paths given the other path one at a time. We will use this at several stages in the proof.

To prove Theorem~\ref{thm:cle_not_determined}, we will first show that, on the positive probability event that $\gamma_1$ and $\gamma_3$ intersect each other, it is not possible to determine whether $\gamma_1$ terminates at $z_2$ or $z_4$ when one just observes the union of the \emph{ranges} of $\gamma_1$ and $\gamma_3$.  
We let $\ul{o}= (o_1,\ldots,o_4 )= (-i,1,i,-1)$ and for each $\delta$, we define~$\CT_\delta$ to be the collection all $4$-tuples $\ul{z}$ where for each~$j$, $| z_j - o_j | <\delta$.

We will simple write $\CT := \CT_{1/100}$ (we choose the value $1/100$ just because it will be small enough for our purposes).  
The results of the previous section show that there exists a positive $p_0 = p_0 (\kappa)$ such that the paths $\ul \gamma$ hookup in each of the two possible ways with $\p_{\ul z}$ probability at least $p_0$ for all $\ul{z} \in \CT$. 

We also denote by $\ul \nu$ the measure on quadruples ${\ul {z}}  {\in} \CT$ obtained by sampling independently each $z_j$ uniformly on the part of the unit circle that is at distance less than $1/100$ from $o_j$.
We define then the law $\p_{\ul \nu}$ which is obtained by first choosing ${\ul z}$ according to $\ul \nu$ and then sampling $\p_{\ul z}$.    

Let us also introduce some further notation that we will use throughout this section. 
We denote by $U(r)= U(r, {\ul \gamma})$ the event 
that all four strands $\gamma_1, \ldots, \gamma_4$ reach the circle of radius $r$ around the origin, and we call $t_j (r)$ their respective hitting times of this circle. 
On this event $U(r)$, we then  define the connected component $D_r = D_r ({\ul \gamma})$  of $ \D \setminus \cup_j \gamma_j([0, t_j (r)])$ that contains the origin, and the conformal transformation $\psi_r=\psi_{r, {\ul \gamma}}$ from $D_r ( {\ul \gamma})$ back onto $\D$ with $\psi_{r, {\ul \gamma}} (0)=0$ and $\psi_{r, {\ul \gamma}}' (0) > 0$. We then also consider the image ${\ul z}(r) = {\ul z} (r, {\ul \gamma})$ under $\psi_{r, {\ul \gamma}}$ of the four endpoints  $\gamma_j (t_j (r))$. The previous considerations show that conditionally on $U(r)$ and on the four strands up to the hitting times
$t_j (r)$, the law of the image ${\ul \gamma}^r$ of the remaining to be discovered parts of the four strands under $\psi_{r, {\ul \gamma}}$ is exactly $\p_{{\ul z} (r, {\ul \gamma})}$.

\subsection{A priori four-arm probability estimates}
\label{subsec:paths_together}
 
The first main purpose of this subsection is to derive  Lemma~\ref{lem:first_moment}, which is a crude lower bound of the probability that the four strands $\gamma_1, \ldots, \gamma_4$ all get close to the origin in a  fairly well-separated way.  We note that our goal here is to prove in a short way a result that will be sufficient for our purpose, and that it would not be difficult to derive somewhat stronger statements. 

When $n \in \N$, we define $\eps_n := 2^{-n}$ (we will use this notation throughout this section) and the event $E_n= U(\eps_n, {\ul \gamma})$ that all four paths $\gamma_1, \dots, \gamma_4$ reach the circle of radius $\eps_n$ around the origin. Let us first point out the following fact:
\begin{lemma}
 \label{prel:lem}
 There exist $\alpha_0 \in (0,2)$ and some constant $c_0 > 0$ such that
$\p_{\ul \nu} [ E_n ] \ge c_0 \times (\eps_n)^{\alpha_0}$  
for all $n \ge 1$.
\end{lemma}

\begin {proof}
We will use here a known estimate about the set of double points of an SLE curve when $\kappa \in (4,8)$. This estimate follows for instance 
 from \cite[Theorem~1]{MW_INTERSECTIONS}, where the almost sure double point dimension of $\SLE_{\kappa}$ is actually derived, but we note that it would also be possible to 
 derive the weaker statement that we will use here in a fairly elementary way.  Indeed, we will just need a rather crude lower-bound of the first moment. 

Suppose that the lemma would not hold, and let us then prove that it would imply an estimate that in turn implies that almost surely, the Hausdorff dimension of the set of  points that are in a certain fixed neighborhood of the origin and on $\gamma_1 \cap \gamma_3$ is almost surely equal to $0$.  This leads to a contradiction, because we know that this is not the case, see for instance \cite[Theorem~1]{MW_INTERSECTIONS}.

First of all, let us suppose that $|z|$ is very small, and consider the conformal transformation $\psi_z \colon \D \to \D$ with $\psi_z(0) = z$ and $\psi_z'(0)> 0$.  We note that $|\psi_z'|$ converges uniformly to $1$ on the closed unit disk when $| z  |  \to 0$, and that $\psi_z'$ converges uniformly to $1$ on the unit circle. 
Let $\p_{\ul \nu'}$ denote the same law as $\p_{\ul \nu}$ except that the four starting points are now chosen uniformly on the set of 4-tuples such that for each $j$, $| z_j - o_j | < 1/200$ (instead of $1/100$), and let $U(\eps_n, {\ul \gamma}, z)$ denote the event that all four paths $\gamma_1, \dots, \gamma_4$ reach the circle of radius $\eps_n/2$ around $z$.
Note that when $r$ is small enough, then the image under $\psi_z$ of the disk of radius $\eps_n$ around $0$ contains the disk of radius $\eps_n /2$ around $z$ (for all $n$), 
and we can also control the image of the uniform measure on the part of the unit circle at distance smaller than $1/100$ of $o_j$ under $\psi_z$ (and see that its density is 
everywhere larger than $(2+ c)$ times that of the uniform uniform measure on the part of the unit circle at distance smaller than $1/100$ for any very small given $c$). 
It then follows readily by conformal invariance that when $r$ is fixed and small enough, then for all $n$ and $|z| < r$,
\[\p_{\ul \nu'} [ U ( \eps_n, {\ul \gamma }, z ) ] \le 32 \p_{\ul \nu} [ E_n ]\]
(where we have chosen $c$ so that $(2+c)^4 =32$).
If we now suppose that the lemma would not hold, then it would imply that for each $\alpha \in (0,2)$, there exists $n_k \to \infty$ such that 
\[ \p_{\ul \nu} [E_{n_k}] / \eps_{n_k}^\alpha \to 0 \quad\text{as}\quad k \to \infty.\]
Hence, 
$$
\sup_{z : | z | < r }  \p_{\ul \nu'} [ U ( \eps_{n_k}, {\ul \gamma }, z )]  = o ( \eps_{n_k}^\alpha ) \quad\text{as}\quad k \to \infty.
$$
This implies an upper bound on the expectation of the area of the set of points in the disk of radius $r$ around the origin that are in the $\eps_{n_k} /2$-neighborhood of both $\gamma_1$ and $\gamma_3$, and one can conclude that, under the probability measure $\p_{\ul \nu'}$, the Hausdorff dimension of $\gamma_1 \cap \gamma_3 \cap \{ z : | z| < r \}$ is at most $2 - \alpha$. As this is true for all $\alpha \in (0,2)$, we conclude that this Hausdorff dimension is almost surely equal to $0$, and as we have already explained, this is not the case.   
\end {proof}

For each $\delta >0$, we then define the event $F_{n, \delta} \subset E_n$ that $\min_{i \neq j} |z_i (\eps_n, {\ul \gamma}) - z_j (\eps_n, {\ul \gamma}) | \geq \delta$. 
In other words, the event $F_{n, \delta}$ says that in terms of harmonic measure from the origin, the four points of ${\ul z} ( \eps_n, {\ul  \gamma})$ are $\delta / (2 \pi)$-separated in $D_{\eps_n} ( {\ul \gamma})$.

\begin{lemma}
\label{lem:first_moment}
There exists $\delta_0 >0$ so that for infinitely many values of $n$,
$\p_{\ul{ \nu}}[ F_{n , \delta_0} ]  \ge c_0 (\eps_n)^{\alpha_0} / 8$.
\end{lemma}
\begin{proof}
Let us first note that, if we choose $\delta_0 >0$ small enough, then for all $n$, 
$ \p_{{\ul \nu }}  [ E_{n+1} | (E_n \setminus F_{n, \delta_0})] \le 1/8 $. Indeed, when $E_n \setminus F_{n, \delta_0}$ holds, then at least two of the strands of ${\ul \gamma}$ corresponding to very close points $z_j (t_j (\eps_n))$ will be very likely to hook up without reaching the circle of radius~$\eps_{n+1}$ (here, we can consider 
separately the following two cases: Either, one point $z_j (t_j (\eps_n))$ is $(\delta_0)^{1/2}$-close to only one other $z_{j'} (t_{j'} (\eps_n))$, in which case the two corresponding strands will hookup with a probability close to $1$ because of Lemma \ref {lem:paths_hookup2}, and will therefore have a very small probability of reaching the circle of radius $\eps_{n+1}$. Or one   $z_j (t_j (\eps_n))$ is $(\delta_0)^{1/2}$-close to at least two other $z_{j'} (t_{j'} (\eps_n))$, in which case, two of the three corresponding strands do necessarily hookup and will have a very small probability of reaching the circle of radius $\eps_{n+1}$ --- we leave the details to the reader).

Suppose now that for such a choice of $\delta_0$, and for all $n$ greater than some $n_0$, 
$ \p_{{\ul \nu}} [ F_{n, \delta_0} ] \le \p_{\ul \nu} [ E_n ] /8$.
Then, we would get that for all $n > n_0$,
\[ \p_{{\ul \nu}}  [ E_{n+1} ] \le \p_{{\ul \nu}}  [ F_{n, \delta_0}] + \p_{{\ul \nu}} [ E_n \setminus F_{n, \delta_0}] / 8 
\le \p_{{\ul \nu}}  [ E_n ] / 4\]
which would imply that $\p_{{\ul \nu}}  [ E_n]$ is bounded by a constant times $4^{-n} = (2^{-n})^2$. But we know from the previous estimate that this is not the case, and we
can therefore conclude that $\p_{{\ul \nu}}  [ F_{n , \delta_0} ] \ge \p_{{\ul \nu}}  [ E_n] / 8 $ for infinitely many values of $n$. 
\end{proof}

We are now going to define a new event by ``composition'' of $F_{n, \delta_0}$ with other events, which is an idea that we will repeatedly use. 
One reason for introducing these other events is that we cannot directly apply the estimate of the lemma to $\ul \gamma^{\eps_n}$ (we will also call this set of paths $\ul f_n (\ul \gamma)$) when $F_{n, \delta_0}$ occurs because we only know in this case that the four points $\ul z ( \eps_n , \ul \gamma)$ are $\delta_0$-separated, while the estimate of the lemma applies for starting points that are distributed according to $\ul \nu$. 

Let us now define three events $G_1$, $G_2$ and $G_3$: 
\begin {itemize} 
\item Suppose that the 4-tuple $\ul z$ is in $\CT_{\delta_0}$. 
We then explore successively each of the four strands of $\gamma$ up to the first time (if it exists) at which 
they reach a capacity (as measured via the log-conformal radius from the origin, as customary for radial Loewner chains) equal to $\delta_0^2$. So we first explore the first strand up to that time, then 
map back, and then grow the image of the second strand up to that time and iterate until we did explore the fourth strand.

If we map back the obtained configuration to the unit disk (using the renormalization of these maps at the origin) provided the four stands did make it till that capacity, the four tips get mapped onto a 4-tuple
$\ul z^1 := \ul z^1 (\ul \gamma)$, and we call $e_1 (\ul \gamma)$ the four strands emanating from $\ul z^1$. 
Note that the conditional law on $e_1 (\ul \gamma)$ is then $\p_{\ul z^1}$. See 
the left part of Figure \ref {sketchAppB2} for a sketch. 
When $\delta_0$ is small enough, it is easy to see (see Appendix \ref {AppB} for some details) that 
the density of $\ul z^1$ on $ \CT_{\delta_0}$ is bounded from below by some positive constant, uniformly with respect to the 4-tuples $\ul z \in \CT_{\delta_0}$. 
By then tossing an additional independent uniform random variable in order to discard some configurations, we see that it is possible to define an event $G_1$ that 
has a positive probability $c_1$ which is independent of $\ul z \in \CT_{\delta_0}$, so that the conditional law of $\ul z^1$ given $G_1$ is exactly the measure $\ul \nu$. 

Note that by choosing $\delta_0$ small enough, we can also ensure that Lemma \ref{lem:first_moment} holds too. 
\begin{figure}[ht!]
\begin{center}
\includegraphics[width=0.35\textwidth]{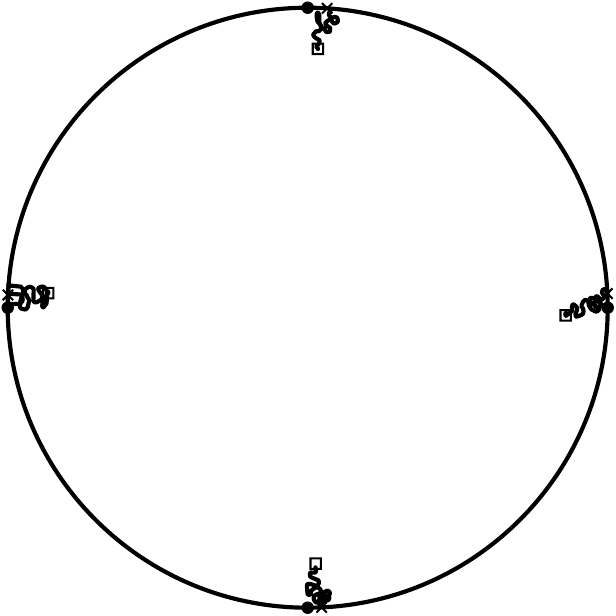} 
\hskip 8mm
\includegraphics[width=0.35\textwidth]{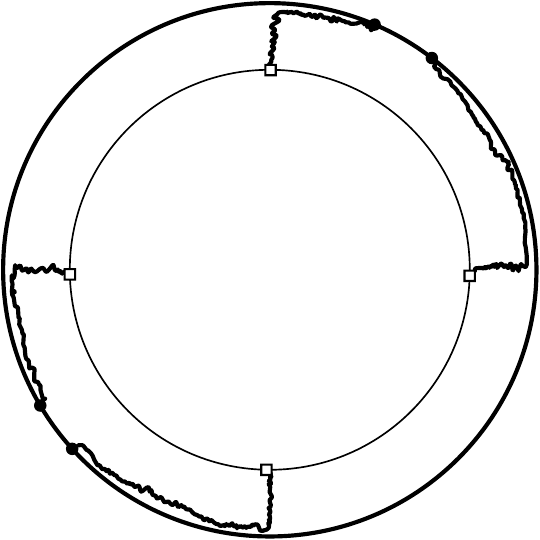}
\end {center}
\caption{\label{sketchAppB2} The events $G_1$ and $G_2$.}\end{figure}

\item We now suppose that $\delta_1 = \delta_0/8$ and that $\delta_0$ is fixed in the way that we have just described.
Consider now a 4-tuple of starting points $\ul z$ that are $\delta_1$-separated, and explore the four strands starting from $\ul z$ up to their 
respective hitting time of the circle of radius $3/4$ around the origin. We denote by $G_2$ the event that $\ul z ( 3/4, \gamma) \in \CT_{\delta_1}$. See Figure \ref {sketchAppB2} for a sketch. 
We claim that the probability of $G_2$ is bounded from below by some constant $c_2$, independently of the choice of $\delta_1$-separated starting points $\ul z$. 
When $G_2$ holds, we note by $e_2 ( \ul \gamma)$ the four strands $\ul \gamma^{3/4}$. 

\begin{figure}[ht!]
\begin{center}
\includegraphics[width=0.35\textwidth]{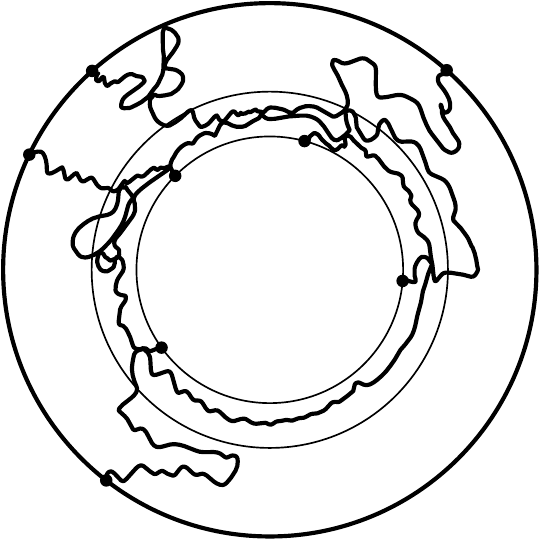}
\end {center}
\caption{\label{sketchAppB3} The event $G_3$.}\end{figure}

\item We again suppose that we start with a $\delta_1$-separated  4-tuple of starting points $\ul z$. We explore the four strands again up to their respective hitting time of the circle of radius $1/2$. 
We call $G_3$ the event that  the four points ${\ul z} (1/2, {\ul \gamma})$ are $\delta_1$-separated and that 
the set $D_{1/2} ({\ul \gamma})$ is a subset of the disk of radius $2/3$ around the origin. In other words, this second condition means that 
the union of the four strands up to their hitting time of the circle of radius $1/2$ do disconnect 
the origin from the circle of radius $2/3$. See Figure \ref {sketchAppB3} for a sketch. 
We now claim that the probability of $G_3$ is bounded from below by some universal constant $c_3$, independently of the choice of $\delta$-separated quadruple of starting points $\ul z$. 
When $G_3$ holds, we then call $e_3 (\ul \gamma) = \ul \gamma^{1/2}$. 
\end {itemize}

The proofs of these three claims (that provide the existence of $c_1$, $c_2$ and $c_3$) use the resampling property that we described before and some elementary distortion estimates. We will indicate in Appendix \ref {AppB} a roadmap 
to derive the statement for the events $G_1$ and $G_3$, and we leave the proofs of the claim for $G_2$ to the interested reader.

From now on, we choose $\delta_0$ small enough such that Lemma~\ref{lem:first_moment} holds and such that the first item (making it possible to define $c_1$) holds. We will just write $F_n$ instead of $F_{n, \delta_0}$. 

Suppose that $\ul z \in \CT_{\delta_1}$. 
We now say that $\ul \gamma$ satisfies the event $E_{n}'$ if the following four events holds: 
(i) The event $G_1$ holds, (ii) the configuration $e_1 ( \ul \gamma)$ satisfies $F_n$, (iii) The configuration $\ul f_n ( \ul e_1 (\ul \gamma)) $ satisfies $G_3$, and (iv) 
the configuration $ \ul e_3 (\ul f_n ( \ul e_1 ( \gamma))) $ satisfies $G_2$.
 We then call $\ul e ( \ul \gamma)$ the obtained configuration i.e, 
$$\ul e ( \ul \gamma ) := [\ul e_2 \circ  \ul e_3 \circ   \ul f_n  \circ \ul e_1]  ({\ul \gamma}).$$ 
Note that when the event $E_n'$ holds,  the starting points of $\ul e ( \ul \gamma)$ are in $\CT_{\delta_1} \subset \CT_{\delta_0}$, so that it will be possible to iterate such events $E_n'$. 

The previous estimates show that there exists constants $c_0, \ldots, c_3$ such that for infinitely many $n$ and 
for any $\ul z \in \CT_{\delta_1}$, 
$$\p_{\ul z} [ E_n' ] \ge c_0 c_1 c_2 c_3 \times (\eps_n)^{\alpha_0}.$$ 
This shows in particular that there exist infinitely many values of $n$ so that for all $\ul z \in \CT_{\delta_1}$, 
\begin{equation}
\label{eqn:e_n_pp_lbd}	
\p_{\ul z} [E_n'] >  (10^4 \eps_n)^{(\alpha_0 +2 )/2}.
\end{equation}
We now choose a value of $N \ge 10$ so that~\eqref{eqn:e_n_pp_lbd} holds. We will keep this $N$ fixed until the end of this section. 

Finally, using some additional randomness in order to discard part of the event $E_N'$, we can define for each $\ul z \in \CT_{\delta_1}$ an event $E_N'' \subset E_N'$ so that 
the value of $\p_{\ul z} [ E_N'']$ does not depend on $\ul z$ (for $\delta_1$-separated $\ul z$) and such that \eqref{eqn:e_n_pp_lbd} still holds. 

We then finally define  $b$, $\beta_0$ and $\beta$ so that 
\[ b:= \p_{\ul z} [E_N''] = (\eps_N)^{\beta_0} = (100 \eps_N)^\beta \]
and note that $0 < \beta_0 < \beta  <  2$.

On the event $E_{N}''$, we define the domain $D_{N}''$ to be the  connected component containing the origin of the complement in the unit disk of the four strands up to the respective times at which one sees that $E_N''$ is satisfied.  Note that the conformal radius (from the origin) of $D_N''$ is in the interval $[4\eps_{N}, \eps_N/4]$ (this follows from multiplicativity of the conformal radii, Koebe's $1/4$ Theorem and from the definitions of  $E_N'$ and $G_1, \ldots, G_3$).

\subsection{The good pivotal regions  and their number} 

We are now going to define the iterated events $E^k$ for $k \ge 1$. The event  $E^1$ is just the event $E_{N}''$ with $N$ chosen as before.  Then, we define iteratively 
for each $k \ge 2$ the event $E^k$ to be the event that $E^1$ holds and 
that $e({\ul \gamma})$ satisfies the event $E^{k-1}$. As on the event $E^1=E_{N}''$, the configuration of the images of the end-points is in $\CT_{\delta_1}$, it follows
immediately that for all ${\ul z} \in \CT_{\delta_1}$ and all $k \ge 1$, 
we have that
\[ \p_{\ul z} [ E^k ]  = b^k . \]

Let us now make some comments of the shape of the connected component $D^k$ containing the origin of the complement of the four strands up to the stopping times corresponding to the event $E^k$ (for instance, $D^1= D_N''$). Let us denote by $\rho_k$ the conformal radius of $D^k$ as viewed from the origin. It follows from our definitions of the event $E_N''$ together with Koebe's $1/4$ Theorem  that: 
\begin {itemize}
 \item For all $k$,  $\eps_N / 4 \le \rho_{k+1} / \rho_k \le 4 \eps_N$. This implies in particular that 
 $\rho_k \le (4 \eps_N)^k$.
 \item The boundary of $D^k$ is included in the annulus between the circles of radii $\rho_k / 4$ (this is just Koebe's $1/4$ Theorem) and $10 \rho_k$ around the origin (this last fact follows readily
 from the disconnection event in the definition of $G_3$). 
\end {itemize}

Suppose now that $u$ is a point in the unit disk. We use again the M\"obius transformation $\psi_u \colon \D \to \D$ with $\psi_u (u) =0$ and $\psi_u'(u) > 0$.  For a given configuration defined under~$\p_{\ul z}$, we say that the event $E^{k} (u)$ holds if the image of the configuration under $\varpsi_u$ satisfies $E^k$. 

We define $\CT' =\CT_{\delta_1 /2 }$. We can note that one can then find $r_0$, so that for all $u$ with $|u| < r_0$, $\psi_u ( {\ul z}) \in \CT_{\delta_1}$ as soon as ${\ul z} \in \CT'$.  Hence, for $|u| < r_0$ and all ${\ul z} \in \CT'$,
$\p_{{\ul z}} [ E^k (u) ] = \p_{\psi_u(  {\ul z}) } [E^k]  = b^k$.

Our next goal is now to derive the following second moment bound:
\begin{lemma}
\label{lem:perfect_two_points}
There exists a constant $C' >0$ so that for all $\ul{z} \in \CT'$, for all $k$ and all $u,v \in B(0,r_0)$ with $|u-v| \ge 2^{-Nk}$, we have
\[ \p_{\ul{z}}[E^{k} (u) \cap E^{k} (v)] \leq \frac { C' \times b^{2k} } { |u-v|^{\beta}}. \] 
\end{lemma}
\begin{proof}
Let us define $D^k (u)$ and $D^k (v)$ just as before, except that they correspond to the domain around $u$ and $v$ respectively (so that $\psi_u (D^k (u))$ has the same law as $D^k$, for instance). Let $K(u,v)$ denote the smallest $k$ such that $D^k (u)$ does not contain the disk of radius $16 (\eps_N)^{-1} |u-v|$ around $u$, and $K (v,u)$ similarly (interchanging $u$ and $v$). By symmetry, it is sufficient to bound 
the probability of the event $E^{k} (u) \cap E^{k} (v) \cap \{ K(u,v) \le K (v,u) \}$. We are going to decompose this according to the value of $K(u,v)$.

Note that by our previous bounds on the conformal radius of $D^k$, 
\[ | u-v | /( 4 \eps_N ) \le \rho_{K (u,v)} \le (4 \eps_N)^{K(u,v)} \]
so that $(4 \eps_N)^{K(u,v)} \ge | u-v | $. We can therefore restrict ourselves to the values $k_0$ taken by $K(u,v)$, so that 
\[ b^{k_0} = (100 \eps_N)^{k_0 \beta} \ge |u-v|^{\beta}.\] 

Suppose that $k_0 = K(u,v) \le K(v,u)$ and  
let us consider the four strands $\gamma_1, \ldots, \gamma_4$ up to the time at which the event $E^{k_0+10} (u)$ is realized. Observing these four strands, we are only missing the 
pieces in $D^{k_0 +10} (u)$, so that can already see what happened near $v$. In particular, we can see -- modulo whether the paths $\gamma_1, \ldots, \gamma_4$ hook up in the right way near $u$ -- 
if $E^{k} (v)$ can hold or not. Furthermore, the conditional probability of the four paths making it so that $E^k (u)$ holds is bounded by a constant times $b^{k - (k_0 +10)}$. 
From this, we can deduce  that 
\begin {align*}
&\p_{\ul z} [ E^{k} (u) \cap E^{k} (v) \cap \{ k_0 =  K(u,v) \le K (v,u) \} ] \\
&\leq  b^{k-k_0 - 10}  \p [  E^{k} (v) \cap \{ k_0 =  K(u,v) \le K (v,u) \}] \\
&\leq b^{-10}  b^k  |u-v|^{- \beta}    \p [  E^{k} (v) \cap \{ k_0 =  K(u,v)\}]. 
\end {align*}
Summing over all possible values of $k_0$, and using the symmetry in $u$ and $v$, we finally get that 
\[ \p[ E^{k} (u) \cap E^{k} (v)] \leq 2b^{-10}  b^{2k} |u-v|^{-\beta}.\] 
\end{proof}

Let $\CN_k = B(0,r_0) \cap (2^{-kN} \Z^2)$.  Let us now define the number $N_k$ of points in $\CN_k$ such that $E_k (u)$ holds.  Our previous moment bounds imply some control on the law of $N_k$ as $k \to \infty$. Recall that~$\beta_0$ is the value chosen so that $b= 2^{-N\beta_0}$.

\begin{lemma}
\label{lem:perfect_pivotals_tight}
 There exist a constant $a >0$ such that for all $\ul z \in \CT'$ and all $k \ge 1$, we have that
\[ \p_{\ul{z}}[a \leq N_k  /  2^{kN (2 - \beta_0)} \leq 1/a  ] \geq a.\]
\end{lemma}
\begin{proof}Let $X=X_k$ denote the random variable $N_k / 2^{kN (2- \beta_0)}$. As $\p_{\ul z} [E^k (u)] = b^k$, we have that
\[ \E_{\ul z} [ X_k ] = 2^{-2kN} 2^{\beta_0 kN} \sum_{u \in  \CN_k } \p_{\ul z} [ E^k (u) ] = 2^{-2kN} \# \CN_k \]
which is bounded from above and from below by positive constants that are independent of $k$ and of ${\ul z} \in \CT'$. 

On the other hand, 
\[ \E_{\ul z}[( X_k )^2] = (2^{-2kN} 2^{\beta_0 kN})^2  \sum_{u, v \in \CN_k} \p_{\ul z} [ E^k (u) \cap E^k (v) ].\]
The sum when $u=v$ is again easily taken care of by the fact that $\p_{\ul z} [E^k (u)] = b^k$, and 
Lemma~\ref{lem:perfect_two_points} takes care of all the terms in the sum for $u \not= v$; we get that for some constant $C$, for all $\ul {z} \in \CT'$ and all $k$,
\[ \E_{\ul z} [ (X_k)^2 ] \le C + C   (2^{-2kN})^2 \sum_{v \not= u  \in \CN_k} |u-v|^{-\beta}\]
which is easily shown to be bounded by some explicit constant independent of $k$ because $\beta \in (0,2)$. 

This information on the first and second moments of $X_k$ then classically imply the lemma.
\end{proof}

\subsection{Rerandomizing configurations in pivotal regions and conclusion of the proof}
\label{subsec:randomizing}

We now complete the proof of Theorem~\ref{thm:cle_not_determined}.  We begin by establishing the following intermediate result.

\begin{lemma}
\label{lem:randomize_pivotal_points}
Let $\gamma$ be the branch of the $\CLE_{\kappa}$ exploration tree from $-i$ to $i$ in $\D$.  The probability that the conditional law of $\gamma$ given the $\CLE_\kappa$ gasket is not supported on a single path is strictly positive.
\end{lemma}
\begin{proof}
To start, we will first explain the reduction from the setting of a $\CLE_\kappa$ to the setting of the law $\p_{\ul{o}}$ as defined earlier.  Let $\wt{\gamma}$ be the time-reversal of $\gamma$.  Let $\tau$ be the first time that $\gamma$ hits $\partial B(-i,1/4)$.  For each $t \geq 0$, let $\D_t$ be the component of $\D \setminus (\gamma([0,\tau]) \cup \wt{\gamma}([0,t]))$ with $\gamma(\tau)$ and $\wt{\gamma}(t)$ on its boundary.  Let $\wt{\tau}$ be the first time $t \geq 0$ that the cross-ratio of the four marked points corresponding to $\gamma(\tau)$, $\wt{\gamma}(t)$, and the most recent times before $\tau$ and $t$, respectively, that $\gamma$ and $\wt{\gamma}$ have hit the counterclockwise segment of $\partial \D$ from $-i$ to $i$, is equal to $1$.  We take $\wt{\tau} = \infty$ if no such times exist.  On the event that $\wt{\tau} <\infty$, the conditional law of the remaining loop ensemble can be described in terms of (a conformal image of) $\p_{\ul{o}}$ so to complete the reduction we need to explain 
why $\p[\wt{\tau} <\infty] > 0$.  This in fact follows since for any deterministic, continuous, simple path $\gamma_0$ in $\ol{\D}$ connecting $-i$ to $i$ and $\epsilon > 0$, there is a positive chance that $\gamma$ stays within distance $\epsilon$ of $\gamma_0$ (viewed as paths modulo parameterization).  Indeed, to be concrete, we can make the particular choice where $\gamma_0$ consists of the concatenation of the line segments $[-i,0]$, $[0,-1]$, $[-1,1-\delta + i\delta]$, $[1-\delta+i\delta,i]$ where $\delta > 0$ is chosen to be very small and we also choose $\epsilon > 0$ very small.

Let $\ul{\gamma}$ be the four strands defined under the law $\p_{\ul{o}}$.  We also let $\Gamma$ be the corresponding loop ensemble and $\Upsilon$ its gasket. We will show that, with positive probability, the conditional probability that $\gamma_1$ terminates at $z_4$ given $\Upsilon$ is in $(0,1)$.  

For each given $k$, consider the Markov chain on $(\Gamma,\ul{\gamma})$ configurations defined as follows:
\begin{itemize}
\item Pick a point $u \in \CN_k$ uniformly at random,
\item Check whether the event $E^{k}(u)$ occurs, and
\item If so, resample the terminal segments of the paths in~$D^k (u)$ and the rest of the $\CLE_\kappa$ in~$D^k(u)$.
\end{itemize}
Note that this chain preserves the joint law of $(\Gamma,\ul{\gamma})$.

We now want to use Lemma~\ref{lem:perfect_pivotals_tight} to see that if we run the chain for $\lfloor 2^{kN (2 - \beta_0)} \rfloor$ steps there is a positive chance (bounded from below uniformly w.r.t.\ $k$) that there exists exactly one single time during these steps at which the chain discovers a point $u \in \CN_k$ where $E^{k} (u)$ occurs and switches the connections near this point.

To see this, we first notice that if we sample $ \lfloor 2^{kN (2 - \beta_0)} \rfloor $ times a uniformly chosen point in $\CN_k$ for a given configuration $\Gamma$, then with a probability bounded uniformly from below (say, larger than some~$a_0$), one hits exactly one point $u$ for which $E^k (u)$ holds and switches it. Then, we need to argue that once this point has been switched and we get a new configuration $\wt \Gamma$, with positive probability, if we sample {the remaining uniformly chosen points in $\CN_k$ (so that we have sampled a total of $2^{kN (2 - \beta_0)}$ such points)}, then we do not find a point~$v$ for which the event $E^k (v)$ occurs for the new configuration.

To justify this, we note that the proof of Lemma~\ref{lem:perfect_two_points} (in particular, the fact that we derived an upper bound on the conditional probability of $E^k(v)$ occurring given $E^k (u)$, regardless of how the paths hook up near $u$), we get that for every $u \in {\mathcal N}_k$, conditionally on the event that $u$ has been picked among the $2^{kN ( 2 - \beta_0)}$ times and on the fact that $E^k (u)$ did hold at that time, the mean number of points $v \in {\mathcal N}_k$ such that $E^k(v)$ holds either before or after the switch near $u$ is bounded by a constant times $b^k \times \# {\mathcal N}_k$. In particular, the conditional probability that this number of points is greater than some explicit large but fixed constant times $b^k \times \# {\mathcal N}_k$ is as small as we want (provided the constant is chosen large enough). And if this number of points is smaller than this constant times $b^k \times \# {\mathcal N}_k$, then the conditional probability that no further changes are made to 
the configuration during the remaining switching attempts is bounded uniformly from below.

Then, if $\wt{\Gamma}$ denotes the resulting loop ensemble and  $\wt{\Upsilon}$ its gasket, the Hausdorff distance between $\Upsilon$ and $\wt{\Upsilon}$ is at most $(8 \eps_N)^k$, while $\wt {\Gamma}$ has changed more dramatically (now $\gamma_1$ hooks up with the other strand).  Therefore sending $k \to \infty$ (and possibly passing to an appropriate subsequence), we get an asymptotic coupling which satisfies the desired property.
\end{proof}

We  now conclude the proof of Theorem~\ref{thm:cle_not_determined} via a zero-one law type argument. 

\begin{proof}[Proof of Theorem~\ref{thm:cle_not_determined}]
Let $\eta$ be the branch of the $\CLE_{\kappa}$ exploration tree from $-i$ to $i$ in the unit disk.  By Lemma~\ref{lem:randomize_pivotal_points}, we know that $\eta$ is not determined by the gasket $\Upsilon$ of $\Gamma$ with probability at least $p \in (0,1]$. We can parameterize $\eta$ as seen from $i$. Since $\eta(t)$ converges almost surely to $i$ as $t \to \infty$, it is clear that for some given large $t_0$, the probability $p(t_0)$ that $\eta$ up to time~$t_0$ is not determined by~$\Upsilon$ is strictly positive. By scaling and conformal invariance, this probability~$p(t_0)$ is independent of~$t_0$. 

For a given fixed~$t_0$, as explained in the $\CLE_\kappa$ description, it is possible to discover simultaneously~$\eta$ up to time~$t_0$ and the $\CLE_\kappa$ loops it traces. In particular, we can trace~$\eta$ up to the first time $s_0$ after $t_0$ at which it will touch the semi-circle from $-i$ to $i$ again, 
and leave the loop that it was tracing at time~$t_0$ in order to branch towards~$i$. At that time $s_0$, the conditional law of the CLE in the remaining to be explored domain with~$i$ on its boundary is just the law of a CLE in this domain. It can in particular be resampled without affecting~$\eta$ up to time~$t_0$.
Hence, after that time $s_0$, the conditional probability that future of~$\eta$ is not determined by the gasket is still~$p$, independently of~$\eta$ up to time~$t_0$.  Hence, we get that $1- p \le (1-p) (1- p (t_0))$.  As $p > 0$, we conclude that $p  =1$.    
\end{proof}

Note that this argument in fact can be adapted to see that the conditional law of $\eta$ given $\Upsilon$ is almost surely non-atomic. 
Indeed, the previous result shows that for some $\lambda <1$, there is a positive probability $a$ that 
the conditional law of $\eta$ up to time $t_0$ given the $\Upsilon$ has no atom of mass greater than $\lambda$. 
Using the conditional independence after $s_0$, if we define $Q(x)$ to be the probability that the conditional law of $\eta$ given $\Upsilon$ has an atom of mass at least $x$, 
we get readily that for all $x \le 1$, 
$$ Q (\lambda x) \le (1-a) Q( \lambda x) + a (Q(x)),$$ which (together with the fact that $Q(1) =0$) implies that $Q(x)=0$ for all positive $x$. 

\section {Derivation of Theorem~\ref{thm:path_not_determined} and of Theorem~\ref{thm:percolation_not_determined}}

We now explain how to adapt the previous ideas in order to derive the other two theorems stated in the 
introduction.

\subsection{Randomness of continuum percolation interfaces}
\label{subsec:percolation_not_determined}

We first give the proof of Theorem~\ref{thm:percolation_not_determined}.

\begin{proof}[Proof of Theorem~\ref{thm:percolation_not_determined}]
As mentioned in the introduction, our proof for CPIs in the case where $p=1$ (i.e.\ where all of the CLE loops have the same label) can in fact be easily adapted to the more general setting of labeled CLE carpets.  So, for simplicity, we will only include the argument in the case that $p=1$.

Fix $\kappa \in (8/3,4)$ and let $\kappa' = 16/\kappa \in (4,6)$.  We suppose that we have a coupling of a~$\CLE_{\kappa'}$ process~$\Gamma'$ and a~$\CLE_\kappa$ process~$\Gamma$ as described at the end of Section~\ref{subsec:percolation_review}.  Just as in Section~\ref{subsec:hookup4_to_8}, we then explore part of the branch $\eta'$ of the $\CLE_{\kappa'}$ exploration tree from $-i$ to $i$ up to a stopping time $t$, and then, starting from its most recent intersection with $\partial \D$ in the counterclockwise direction starting from $-i$, we start drawing the currently explored loop backwards up to some stopping time, say $\wt{t}$.  Call this curve $\wt{\eta}'$.  We then condition on all of the $\CLE_\kappa$ loops which intersect the part of the $\CLE_{\kappa'}$ exploration tree that we have observed so far.  See the left side of Figure~\ref{fig:percolation_randomness} for an illustration.  By \cite[Theorem~7.3]{cle_percolations}, we know in particular that the conditional law of the~$\CLE_{\kappa'}$ 
exploration tree in the connected component of the remaining domain $D_0$ which has~$i$ on its boundary, is given by that of an independent $\CLE_{\kappa'}$ exploration tree in the remaining domain.  (In the FK-Potts analogy in a domain~$D$, one considers a monochromatic with color ``$A$'' boundary condition for Potts that one can view as a wired boundary condition for the coupled FK model, and one explores the inner boundaries that touch $\partial D$ of the open FK-cluster up to some time -- this corresponds to the exploration of the $\CLE_{\kappa'}$ tree. Then, one attaches to its right all the  clusters of ``non-$A$'' sites which therefore have $A$'s on their outer boundary, and notes that in $D_0$, one has again $A$-monochromatic boundary conditions.)  This exploration tree in turn determines a $\CLE_{\kappa'}$ process $\Gamma_0'$ in $D_0$.  We emphasize that the loops of $\Gamma_0'$ are not necessarily loops of $\Gamma'$ even though the exploration tree associated with the former is a subset of the 
exploration tree of the latter.  (This has to do with how the loops are defined from the exploration tree when the branches are bouncing off the domain boundary.)

This allows us to set things up in a manner which is similar to the proof of Theorem~\ref{thm:cle_not_determined}.  Namely, we continue exploring $\eta'$ and $\wt{\eta}'$ in $D_0$.  Note that these two curves correspond to branches of the exploration tree of $\Gamma_0'$.  Thus, as explained in the proof of Theorem~\ref{thm:cle_not_determined}, we can choose stopping times $\tau$, $\wt{\tau}$ for $\eta$, $\wt{\eta}'$ so that it is a positive probability event that the cross-ratio of the four marked points in $D_0$ corresponding to $\eta'(\tau)$, $\wt{\eta}'(\wt{\tau})$, and the most recent intersections of $\eta'$, $\wt{\eta}'$, respectively, before the times $\tau$, $\wt{\tau}$ with the counterclockwise segment of $\partial D_0$ from $\eta'(t)$ to $\wt{\eta}'(\wt{t})$ (black dots on the right of Figure~\ref{fig:percolation_randomness}) is equal to $1$.  The argument of the proof of Theorem~\ref{thm:cle_not_determined} implies that it is a positive probability event that the conditional law of the remaining 
loop ensemble is not supported on a single configuration.  Let $\eta_0'$ (resp.\ $\wt{\eta}_0'$) be the time-reversal of the loop being explored by $\eta'$ (resp.\ $\wt{\eta}'$) at time $\tau$ (resp.\ $\wt{\eta}$) starting from the most recent intersection of $\eta'$ (resp.\ $\wt{\eta}'$) with the counterclockwise part of $\partial D_0$ from $\eta'(t)$ to $\wt{\eta}'(\wt{t})$ before time $\tau$ (resp.\ $\wt{\tau}$).  On the event that $\eta'$ and $\wt{\eta}'$ form part of the same loop in $\Gamma_0'$, it follows that the initial segments of $\eta_0'$ and $\wt{\eta}_0'$ are not part of the outermost $\CLE_{\kappa'}$.  On the other hand, on the event that $\eta'$ and $\wt{\eta}'$ do not form part of the same loop in $\Gamma_0'$, it follows that the initial segments of $\eta_0'$ and $\wt{\eta}_0'$ are part of the outermost $\CLE_{\kappa'}$.

\begin{figure}[ht!]
\begin{center}
\includegraphics[scale=0.85,page=10]{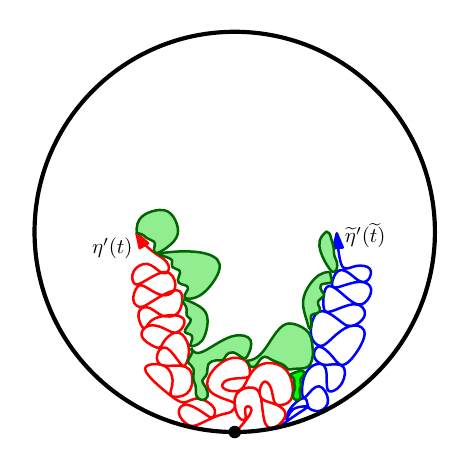}\includegraphics[scale=0.85,page=11]{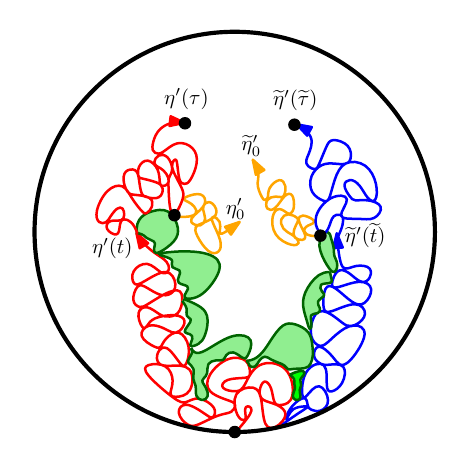}
\end{center}
\caption{\label{fig:percolation_randomness} {\bf Left:} The branch $\eta'$ of the $\CLE_{\kappa'}$ exploration tree from $-i$ to $i$ drawn up to a given time $t$ (red), as well as the time-reversal $\wt{\eta}'$ drawn up to a given time $\wt{t}$ of $\eta'$ targeted at its most recent intersection with the counterclockwise segment of $\partial \D$ from $-i$ to $i$ (blue), just like in Section~\ref{subsec:hookup4_to_8}.  We take the $\CLE_{\kappa'}$ to be coupled with a $\CLE_\kappa$ as a percolation in the $\CLE_\kappa$ carpet; the filled loops are the $\CLE_\kappa$ loops which intersect the part of the $\CLE_{\kappa'}$ we have explored so far.  The conditional law of the $\CLE_{\kappa}$ in the unexplored region is then independently a $\CLE_{\kappa}$ in each of the components.  {\bf Right:} Shown is the continuation of $\eta'$, $\wt{\eta}'$ up to stopping times $\tau$, $\wt{\tau}$ on the event that the cross-ratio of the four marked points (black dots) is equal to~$1$ together with the time-reversals (orange) 
$\eta_0'$, $\wt{\eta}_0'$ of the loops of the $\CLE_{\kappa'}$ associated with the exploration tree restricted to the remaining domain being explored by $\eta'$, $\wt{\eta}'$ at the times $\tau$, $\wt{\tau}$.  If $\eta'$ hits $\wt{\eta}'(\wt{\tau})$ before connecting up with $\eta_0'$ then $\eta_0'$, $\wt{\eta}_0$ are not part of an outermost $\CLE_{\kappa'}$ loop and otherwise they both are.}
\end{figure}
 Then we have four marked boundary points and we can proceed in this setting with the same argument as in the proof of Theorem~\ref{thm:cle_not_determined}.  In particular, the argument of Theorem~\ref{thm:cle_not_determined} implies that if the inner/outer paths create special intersection points, then we cannot tell if the inner paths hook up with each other and the outer paths with each other or if the inner and outer paths hookup.  Note that the actual gasket of the $\CLE_{\kappa'}$ is then not the same depending on the way in which the paths hookup (which is a stronger statement than just saying that the exploration path is not the same).

This implies that the percolation exploration of the $\CLE_\kappa$ is with positive probability not determined by the $\CLE_\kappa$ carpet.  A simple zero-one argument (by looking at smaller and smaller pieces of the CPI) then implies that it is in fact the case that the percolation exploration is almost surely not determined by the $\CLE_\kappa$ carpet and in fact the conditional law is almost surely non-atomic.
\end{proof}

\subsection{Randomness of the SLE curve given its range}
\label{subsec:path_not_determined}

We now turn to  Theorem~\ref{thm:path_not_determined}.  In the present section, we again assume that $\kappa \in (4,8)$.  Let us first note the following fact, that allows us to consider an $\SLE_\kappa(\kappa-6)$ instead of an $\SLE_\kappa$. 

\begin{lemma}
\label{lem:not_determined_same}
The probability that the conditional law of an $\SLE_\kappa$ process given its entire range is non-trivial is either equal to $0$ or to $1$, and it is equal  to the corresponding probability for an $\SLE_\kappa (\kappa -6)$ process.
\end{lemma}
\begin{proof}
By \cite[Proposition~7.30]{MS_IMAG} the conditional law of an $\SLE_\kappa(\kappa-6)$ process given its left and right boundaries is independently that of an $\SLE_{\kappa}(\kappa/2-4;\kappa/2-4)$ in each of the bubbles formed by the left and right boundaries. Note that there are almost surely infinitely many such bubbles, because there are infinitely many global cut points on an $\SLE_\kappa (\kappa -6)$ (this follows from the fact that the same is true for an $\SLE_\kappa$, see \cite {MW_INTERSECTIONS} and the references therein). Note that these left and right boundaries \emph{are} determined by the range of the path.  Therefore, the probability that   an $\SLE_{\kappa}(\kappa-6)$ process is determined by its range is equal to $1$ if the same is true for an $\SLE_{\kappa}(\kappa/2-4;\kappa/2-4)$, and it is equal to $0$ otherwise.

But we know that the conditional law of an $\SLE_\kappa$ process given its left and right boundaries is also independently that of an $\SLE_{\kappa}(\kappa/2-4;\kappa/2-4)$ in each of the infinitely many bubbles formed by the left and right boundaries. Hence, we also have that the probability that   an $\SLE_{\kappa}(\kappa-6)$ process is determined by its range is equal to $1$ if the same is true for an $\SLE_{\kappa}(\kappa/2-4;\kappa/2-4)$, and it is equal to $0$ otherwise. 

This proves the statement in the lemma. 
\end{proof}

We are now ready to prove Theorem~\ref{thm:path_not_determined}.

\begin{proof}[Proof of Theorem~\ref{thm:path_not_determined}] 
By the previous lemma, we can consider the case where $\eta$ is an $\SLE_{\kappa}(\kappa-6)$ process in $\D$ from $-i$ to $i$ which is given as the branch of a $\CLE_{\kappa}$ exploration tree in $\D$.
As explained in the introduction, and in contrast to the proof of Theorem~\ref{thm:cle_not_determined}, we will need to resample the configuration in \emph{two} different well-chosen regions rather than just at \emph{one}, in order to globally preserve the range of $\eta$.

Recall from the branching-tree construction of $\CLE_\kappa$ that the path $\eta$ consists of the (closure of the) concatenation of its excursions away from the counterclockwise arc $\partial$ from $-i$ to $i$, and that each of these excursions is part of a different $\CLE_\kappa$ loop. If one stops $\eta$ at a time $T$ such that $\eta(T) \notin \partial \D$, then the conditional law of the remainder of the $\CLE_\kappa$ loop that it is tracing at that time $T$ is that of an $\SLE_\kappa$ from $\eta (T)$ to $\eta (S)$ in the complement of $\eta [0,T]$ where $S$ is the last time before $T$ at which $\eta$ was in $\partial \D$.   

We suppose that~$r > 0$ is chosen to be very small, and we let~$B_1$ and~$B_2$ be the open disks of radii~$r$ around $-1/2$ and around~$1/2$ respectively.   Fix $k \in \N$.  We consider the following procedure.
\begin{itemize}
\item We condition on the part of $\eta$ and of its time-reversal  (starting from $i$ and $-i$) 
stopped upon hitting $\partial B_1$, and (on the event where they both hit $\partial B_1$) on the time-reversals of the two loops being drawn by each branch when hitting $\partial B_1$, and also stopped upon hitting $\partial B_1$. 
Note that all these branches are all deterministic functions of the $\CLE_\kappa$. We let $E_1$ be the event that these four branches do hit $B_1$. 
\item When $E_1$ does not hold, one does not do anything. On the event $E_1$, we let $\varphi^1$ be the unique conformal transformation from $\D$ to the complementary component $U_1$  of the traced paths, which contains $B_1$ and normalized by $\varphi^1(0) = -1/2$ and $(\varphi^1)'(0) >0$. By our definition of the $\p_{\ul {z}}$,  
 the conditional law of the configuration in $U_1$ is given by the image under $\varphi^1$ of $\p_{\ul{z}}$ for some $\ul{z}$.  As in the proof of Lemma~\ref{lem:randomize_pivotal_points}, we perform the Markov step where we pick $\lfloor 2^{kN \beta_0} \rfloor$ points uniformly in $2^{-kN}\Z^2 \cap B(0,r_0)$, rerandomizing each good pivotal place that we find. The proof of Lemma~\ref{lem:randomize_pivotal_points} implies that, on $E_1$, there is a positive probability bounded from below (by a constant depending only on the cross-ratio of the marked points $\ul{z}$) that the hookup has been changed at exactly one of the attempts, so that the only difference between  the new configuration and the initial one is a small disk (with radius that goes to $0$ as  $k$ gets large).
\item We then repeat the same procedure starting from the obtained configuration, just replacing $B_1$ by $B_2$.
\end{itemize}
By definition, the law of a $\CLE_\kappa$ is invariant for this procedure.  In particular, the law of $\eta$ is invariant for this procedure.

Our goal now is to show that in the limit as $k \to \infty$, this procedure gives us an asymptotic coupling between the original $\CLE_\kappa$ (and its branch of the exploration tree $\eta$) 
and a new $\CLE_\kappa$ such that with positive probability the new branch of the exploration tree from $-i$ to $i$ has the same range as $\eta$ but visits the points of its range in a different order. 

\begin{figure}[ht!]
\begin{center}
\includegraphics[scale=.85]{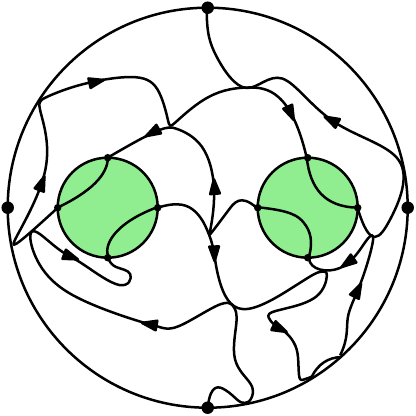}\hspace{0.025\textwidth}\includegraphics[scale=.85]{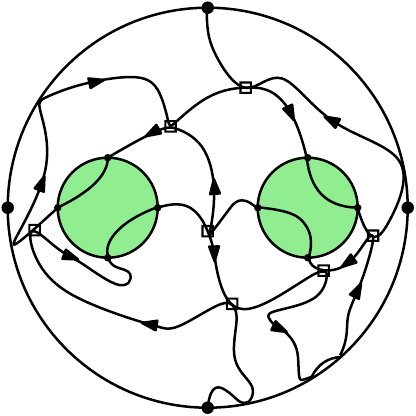}
\end {center}
\caption{\label{fig:eventA} Sketch of the event $A$.  {\bf Left:} The path $\eta$ from $-i$ to $i$.  
{\bf Right:} The boxes indicate the conditions on the branches that have to intersect. }
\end{figure} 

\begin {figure}
\begin {center}
\includegraphics[scale=.85]{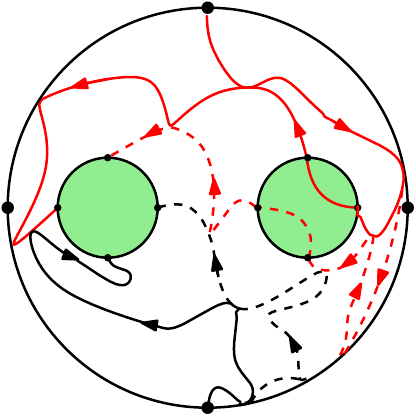}\hspace{0.025\textwidth}\includegraphics[scale=.85]{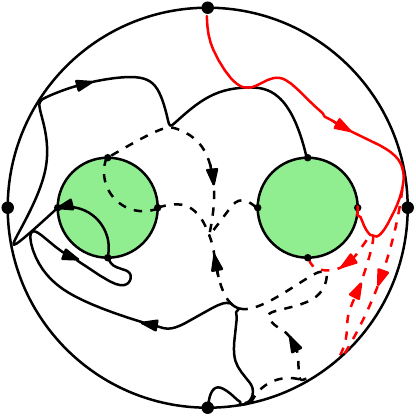}
\end{center}
\caption{\label{fig:Aresampling} Illustration of the two Markovian steps. 
{\bf Left:} Partial exploration of $\eta$ and of its time-reversal before the first Markovian step, and then in dashed the two time-reversed branches,  leading to the definition of $U_1$.  
{\bf Right:} After a successful first Markovian step in $B_1$, the configuration before the second step leading to the definition of $U_2$.  
(Remark: for the Markovian steps to be ``successful'' the two arcs in $U_1$ and then in $U_2$ would typically not be disjoint, but they are drawn as if there are disjoint here  in order to visualize better the change in the hook-up configurations.)}
\end{figure}

We let $A$ be the event that $\eta$ visits (in order) $B_1$, the counterclockwise part of $\partial \D$ from $-i$ to $i$, $B_2$, $B_1$, the clockwise part  of $\partial \D$ from $-i$ to $i$ and $B_2$, and that
furthermore, the various parts of $\eta$ do intersect each other so that they disconnect $B_1$ from $B_2$ as well
as $B_1,B_2$ from $\partial \D$, as schematically depicted in Figure~\ref{fig:eventA}. This event $A$ has a positive probability.  

When $A$ holds, it implies in particular that $E_1$ holds. 
If one performs the Markov step in $U_1$ as described above and ``succeeds'' (meaning that one changed the hook-up in $U_1$ by just changing the configuration 
in one small disk),
then one notes that obtains a new configuration for which 
the event $E_2$ (defined just as $E_1$ but exchanging the role of $B_1$ and $B_2$) holds (see Figure \ref {fig:Aresampling}). In particular, with positive probability, when one performs the second Markov step for this new configuration, the 
probability of changing the hook-up in $B_2$ with exactly one successful modification will be bounded from below as well (and this bound will depend only on the cross-ratio of the corresponding four points, which does in fact not depend on the details of the first resampling procedure, and therefore does not depend on $k$ but only on the initial path $\eta$).

\begin {figure}
\begin {center}
\includegraphics[scale=.85]{figures/two_pivotalsw1}\hspace{0.025\textwidth}\includegraphics[scale=.85]{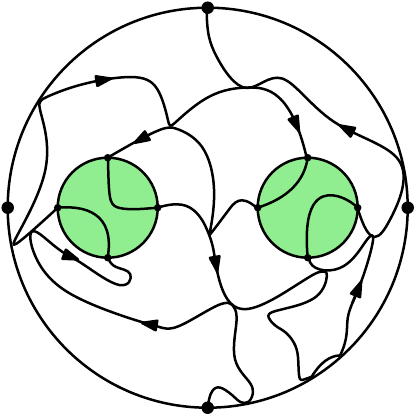}
\end{center}
\caption{\label{fig:Aresampling2} The path $\eta$ before the two Markovian steps, and the resulting new path. (Note that when $k$ is large, the Hausdorff distance between the two paths within $B_1$ and also within $B_2$ are in fact very small, but the hook-up is changed as schematically indicated in these figures). 
}
\end{figure}

Wrapping up, we see that on the event $A$, there is a probability  which is bounded from below uniformly with respect to $k$ that the resampling procedure will change the hookup in $B_1$ and the hookup in $B_2$, and that the initial path $\eta$ gets changed into another nearby path that traces the branches in different order (see Figure \ref {fig:Aresampling2}). 
The result thus follows by taking a limit as $k \to \infty$.

\end{proof}

\section{Comments}
\label{sec:discussion}

\subsection{Relationship with the $\SLE$/GFF coupling}
\label{subsec:gff}

$\SLE_{\kappa}$ and $\CLE_\kappa$ can be naturally coupled with an instance $h$ of the Gaussian free field (GFF) on a simply connected domain $D \subseteq \C$ with appropriately chosen boundary data (see e.g.\ \cite{SchrammShe10,SHE_WELD,DUB_PART,MS_IMAG,MS_IMAG4,cle_percolations}).  Theorem~\ref{thm:path_not_determined} has some consequences for the coupling of $\SLE_\kappa$ for $\kappa \in (4,8)$ with the GFF, Theorem~\ref{thm:cle_not_determined} for the coupling of $\CLE_\kappa$ for $\kappa \in (4,8)$, and Theorem~\ref{thm:percolation_not_determined} for $\CLE_\kappa$ for $\kappa \in (8/3,4)$.

Let us first comment on the $\SLE_\kappa$/GFF coupling for $\kappa \in (4,8)$.  Suppose that $h$ is a GFF on a simply connected domain $D \subseteq \C$ with boundary data so that it may be coupled with an $\SLE_\kappa$ process $\eta$ from one point on $\partial D$ to another.  In this coupling, the boundary data for the conditional law of $h$ given $\eta$ is in each component $U$ of $D \setminus \eta$ given by a constant plus a multiple of the argument of the derivative of the uniformizing conformal map $\varphi \colon U \to \h$.  Although the winding of $\partial U$ is not defined in the usual sense as it is fractal, $\arg \varphi'$ has the interpretation of being the \emph{harmonic extension} of the winding of $\partial U$ from $\partial U$ to $U$.  In particular, there is a marked point on $\partial U$ where $\arg \varphi'$ makes a jump of size $2\pi$.  In terms of the path, this point corresponds to the first (equivalently last) point on $\partial U$ visited by $\eta$.  If one observes only the range of 
$\eta$ in addition to the GFF boundary data then it is in fact possible to recover the trajectory of $\eta$ in a measurable way.  This follows because $\eta$ turns out to be a deterministic function of $h$ \cite{DUB_PART,MS_IMAG} and the values of $h$ in the components of $D \setminus \eta$ are conditionally independent of $\eta$ itself given the values of $h$ along $\eta$.  Theorem~\ref{thm:path_not_determined} therefore implies that one cannot recover the marked points or GFF field heights by observing the range of $\eta$ and the orientations of the loops that it makes alone.

The case of $\CLE_\kappa$ for $\kappa \in (8/3,8) \setminus \{4\}$ is similar to that of $\SLE_\kappa$.  The reason for this is that one couples $\CLE_\kappa$ for $\kappa \in (4,8)$ with the GFF by coupling the whole exploration tree of $\SLE_\kappa(\kappa-6)$ processes with the GFF \cite{MS_IMAG,MS_IMAG4,cle_percolations}.  The case that $\kappa = 4$ is different because in this case the conditional law of the GFF given the loops is given by a constant which is determined by the loop orientations. In particular, the loops are not marked by a special point so that there is no additional randomness involved here.

The Markov step used to prove Theorems~\ref{thm:path_not_determined}--\ref{thm:percolation_not_determined} is also interesting to think about in the context of the GFF: While this operation only affects the small regions of the $\CLE$ picture, it does have a less localized influence for the corresponding GFF.  This is because changing the manner in which the loops of a $\CLE$ are hooked up has the effect of moving the marked point along the component boundaries which in turn translates into changing the GFF heights along the loop boundaries.  That is, our Markov step is a measure preserving transformation defined on GFF instances which leaves the $\CLE$ gasket fixed but makes a macroscopic change to the corresponding GFF instance because the heights are changed.

\subsection{Quantum gravity perspective}
\label{subsec:qg}

It is natural to wonder whether techniques involving quantum gravity and mating of trees, as described e.g.\ in \cite{dms2014mating}, could be used to give an alternate proof of Theorems~\ref{thm:path_not_determined}--\ref{thm:percolation_not_determined}.  In this short subsection, we make some brief and  informal remarks about how the operations described in this paper could potentially be understood and studied within that framework.  The re-randomization procedure that we have described here also naturally fits into the quantum gravity framework developed in \cite{dms2014mating}.  In particular, it is implicit in the constructions of \cite{dms2014mating} that there is a quantum version of the ``natural'' measure on $\SLE_{\kappa}$ double points and intersections of $\CLE_{\kappa}$ loops.  It is not difficult to see that if one picks a typical such point using this measure in either setting and then ``zooms in,'' the resulting limit is the same if one starts in either the $\SLE_{\kappa}$ double point or $\
CLE_{\kappa}$ loop intersection settings.  In fact, it can be described as a gluing of eight so-called quantum wedges which correspond to the four strands of path and the four regions which separate the path strands.  The operation of resampling how the paths are hooked up has a natural interpretation in the quantum gravity perspective.  Indeed, it is shown in \cite{dms2014mating} that an $\SLE_{\kappa}$ path or $\CLE_{\kappa}$ path for $\kappa \in (4,8)$ can be represented as a gluing of a pair of loop-trees which arise from a pair of independent $(\kappa/4)$-stable L\'evy processes; see \cite[Figure~1.6 and Figure~1.7]{dms2014mating}.  These loop-trees correspond to the components which are cut off on the left and right sides of the path.  Regluing the paths in order to switch the direction of a pivotal point corresponds to natural operations that one can preform directly on the trees (hence L\'evy processes) themselves, namely cutting the pair of trees to form new trees or grafting trees together.

\appendix

\section{Bi-chordal resampling}
\label{sec:bichordal}

The purpose of this appendix is to explain how the bi-chordal resampling arguments developed in \cite[Section~4]{MS_IMAG2} can be extended to the setting in which the paths in question are allowed to intersect each other. We write this as an appendix as it can be read independently of the present paper and may also serve as a future reference for such resampling questions. 
First, in Section~\ref{subsec:bc_two_paths}, we will establish a general statement which is used in the proof of Lemma~\ref{lem:sle_k_k_minus_6_reversal} in the case that $\kappa \in (6,8)$ and in Section~\ref{subsec:bc_three_paths} a variant of this that is used in the proof of Lemma~\ref{lem:sle_k_k_minus_6_reversal} in the case $\kappa \in (4,6)$.

\subsection{Basic characterization}
\label{subsec:bc_two_paths}

\begin{theorem}
\label{thm:bc_two_paths}
Fix $\kappa > 0$, $\rho_1,\rho_2,\rho_1',\rho_2' > (-2) \vee (\kappa/2-4)$.  Suppose that $D \subseteq \C$ is a Jordan domain and $x_1,x_2,y_1,y_2 \in \partial D$ are distinct and given in counterclockwise order.  There is at most one probability measure on pairs of paths $\eta_L,\eta_R$ respectively connecting $x_1$ to $y_1$ and $x_2$ to $y_2$ such that the conditional law of $\eta_L$ given $\eta_R$ is independently that of an $\SLE_\kappa(\rho_1;\rho_2)$ process in each of the components of $D \setminus \eta_R$ which are to the left of $\eta_R$ and the conditional law of $\eta_R$ given $\eta_L$ is independently that of an $\SLE_\kappa(\rho_1';\rho_2')$ process in each of the components of $D \setminus \eta_L$ which are to the right of $\eta_L$. 
\end{theorem}

We emphasize that we do not prove the existence of a probability measure which satisfies the hypotheses of Theorem~\ref{thm:bc_two_paths}.  However (and this is the case in the present paper's $\CLE_\kappa$ setting or in the imaginary geometry framework), the existence of such a measure is typically provided by ad-hoc constructions.

We note that it suffices to prove Theorem~\ref{thm:bc_two_paths} in the case that $\kappa \in (0,4]$ because of $\SLE$-duality considerations.  More precisely, in the case that $\kappa > 4$ we can replace~$\eta_L$ by its right boundary and~$\eta_R$ by its left boundary because the conditional law of~$\eta_L$ and~$\eta_R$ in each case is known (see, e.g., \cite{MS_IMAG}).

Throughout, we will make the concrete choice that $D = (0,1) \times (0,\ell)$ for some fixed $\ell > 0$.  We denote the left boundary segment $[0, i \ell]$  by $L$ and the right boundary segment by~$R$.  Let $\CK_L$ (resp.\ $\CK_R$) denote the set of all compact connected subsets of~$\overline D$ that contain $L$ (resp.\ $R$). We endow $\CK_L$, $\CK_R$ with the Hausdorff distance between compact sets and consider them with their associated Borel $\sigma$-algebra.

We are now going to give a definition of a measurable family of elements of $\CK_L$ that we will call good, but we already note that if $K$ consists of the union of $L$ with a continuous curve from $0$ to $i \ell$ in $\overline D$, then it will necessarily be such a good set.

When $K_1 \in \CK_L$, let us look at the connected components of $D \setminus K_1$ whose boundary contains a non-trivial interval of $R \setminus K_1$. We call $(O_j)$ this family of open sets, and we denote by $a_j$ and $b_j$ the corresponding extremities of this interval of $(R \setminus K_1) \cap \overline O_j$. 

For each $n$, we also define the set $K_1^n$ which is obtained by taking the union of all closed dyadic squares of side-length $2^{-n}$ (i.e., with corners in $2^{-n} \Z^2$)  that intersect $K_1$.  For each $n$, the map $K_1 \mapsto K_1^n$ is easily shown to be measurable and it can take only finitely many values. Hence, any function of $K_1^n$ will be a measurable function of $K_1$. Furthermore, $K_1^n$ is also in $\CK_L$ and we can then define the family $(O_i^n)$ and the points $a_i^n$ and $b_i^n$ just as before.

\begin {definition}
We say that a set $K_1$ is good (in $\CK_L$) if the following two conditions hold: 
\begin{enumerate}[(i)]
\item\label{it:good1} For each $N > 0$, for each $n$, the number of $(O_i^n)$ with diameter greater than $1/N$ is finite and bounded uniformly in $n$. 
\item\label{it:good2} For each $j$, the boundary of $O_j$ is a continuous curve. 	
\end{enumerate}
\end {definition}

Note that since each $O_i^n$ is contained in some $O_j$, and that each $O_j$ contains at least one $O_i^n$ of half its diameter when $n$ is sufficiently large, this implies also that for each $N$, the number of $O_j$ with diameter greater than $1/N$ is finite.

Let us give some further definitions: For each $O_j$ and each large $n$, let us denote by $O_{i_n (j)}^n$ the largest $O_i^n$ that is contained in $O_j$ (for instance, the one with largest diameter, breaking possible ties using some deterministic rule) if it exists (and it always does provided $n$ is large enough). Let us write $\overline a_j^n:= a_{i_n (j)}^n$ and $\overline b_j^n := b_{i_n (j)}^n$ the corresponding boundary points.

A first remark is that the set of good sets is a measurable subset of $\CK_L$.  Indeed, both~\eqref{it:good1} and~\eqref{it:good2} can be expressed in terms (i.e., as countable unions of intersections of unions of intersections) of events involving finitely many sets $K_1^n$.  This is clear for~\eqref{it:good1}. To see that this is also the case for~\eqref{it:good2}, one can note that~\eqref{it:good2} is equivalent to the fact that for each $j$, there exists parameterizations of the boundaries of $O_{i_n (j)}^n$ (which are continuous curves) such that this sequence of continuous curves is uniformly Cauchy as $n \to \infty$.  This can be seen by considering the sequence of conformal maps $\phi_n^j$ which take $\D$ to $O_{i_n(j)}^n$ with $-i,-1,i$ respectively taken to $\ol{a}_j^n$, $(\ol{a}_j^n+\ol{b}_j^n)/2$, $\ol{b}_j^n$ and noting that they are uniformly Cauchy if and only if~\eqref{it:good2} holds.

For each given $K_1 \in \CK_L$, we now describe a procedure $\Phi$ to define a random element $K_2 \in \CK_R$: First, if $K_1$ is not good, we just set $K_2 = R$.  If $K_1$ is good, then inside each of the sets $O_j$ (using the same notations as above), we then sample an independent $\SLE_\kappa(\rho_1'; \rho_2')$ denoted by $\gamma_j$ from $a_j$ to $b_j$, with marked points that are both located at $a_j$ (one on each ``side''). We then define the set $K_2 :=R  \cup{ ( \cup_j  \gamma_j)}$. Note that this set is compact, connected and contains $R$, i.e., it is in $\CK_R$. 

A first important observation is the following: 

\begin{lemma}
If we endow $\CK_L$ and $\CK_R$ with the Hausdorff metric, the previous procedure describes a Borel-measurable Markov kernel (in other words, the map $\Phi$ that associates to each $K_1$ the law of $K_2$ is measurable). 
\end {lemma} 

\begin {proof}
It suffices to show that the law of $K_2$ can be viewed as the weak limit as $n \to \infty$ 
of  $\Phi (K_1^n)$ (which are therefore measurable with respect to $K_1$):

When $K_1$ is good, for each $n$, let us define $\gamma_i^n$ the corresponding SLE curves in $O_i^n$. For each $j$, condition~\eqref{it:good2} shows that the law of $\gamma_{i_n (j)}^n$ converges to that of $\gamma_j$.
Indeed, if we define the conformal map $\psi_n^j$  from $O_j$ onto $O_{i_n (j)}^n$ that maps $a_j, (a_j + b_j)/2, b_j$  onto $\overline a_j^n, (\overline a_j^n+ \overline b_j^n) /2, \overline b_j^n$ respectively, then because $K_1$ is good, then $\psi_n^j$ extends continuously to the boundary (i.e. to the set of prime-ends) and converges uniformly to the identity map as $n \to \infty$. 

Now we note that for each $\eps > 0$, the number of connected components $O_i^n$ with diameter greater than $\eps$ is equal to the number of connected component $O_j$ with 
diameter greater than $\eps$ for all $n$ large enough (only the $O_{i_n(j)}^n$ will have diameter greater than $\eps$  (which follows readily from condition~\eqref{it:good2}).

From this, it follows readily that when $n$ is large enough, one can couple $K_2^n$ (defined just as $K_2$ but replacing $K_1$ by $K_1^n$) and $K_2$ so that the Hausdorff distance between the two is smaller than $2 \eps$ (as one just needs to control the SLE paths in the finitely many $O_j$'s). Hence, the law of $K_2$ can indeed be viewed as the weak limit of the law of $K_2^n$. 
\end {proof} 

In a completely symmetric manner (with respect to the line $1/2 + i \R$), we can define for each $K_2 \in \CK_R$ a procedure to define a random set $K_1' \in \CK_L$, just replacing $(\rho_1'; \rho_2')$ with $(\rho_1; \rho_2)$.  We denote by $\Psi$ the measurable Markovian kernel that associates to each $K_1 \in \CK_L$ the law of the random set $K_1'$ obtained iteratively by first choosing $K_2$ applying the first procedure, and then $K_1'$ (in a conditionally independent way, given $K_2$) using this second procedure. This kernel is Markovian as a composition of Markovian kernels.   One then has the following result:

We now state the following proposition that immediately implies Theorem~\ref{thm:bc_two_paths}. 
\begin{proposition}
\label{prop:markov_kernel}
There exists at most one probability measure $\pi$ on $\CK_L$ that is invariant under $\Psi$ and such that $\pi$ a.e.\ $K_1$ is the union of the range of a continuous curve and $L$. 
\end {proposition}

This is useful in the present paper, because in the CLE/SLE settings that we work with, we are given a probability measure on $\CK_L$ that satisfies these properties. 

\begin{remark}
Stronger statements probably do hold, but the previous one is often sufficient (and it is the case in the present paper). For instance, one can probably show uniform exponential mixing (i.e., that starting from any two different configurations,  the two chains started from these two configurations can be coupled so that they coincide before the $N$-th iteration step with a probability that is bounded by $\exp ( -cN)$ for some constant $c$ that is independent of the initial configurations). 
\end{remark}

Recall some basic features of Markov kernels (see for instance Chapter 6 of \cite {Varadhan}) that will be useful in the proof: 
The set of invariant probability measures is convex, and the extremal points in this convex set are exactly the ergodic invariant measures.  As a consequence, two extremal ergodic invariant measures are either mutually singular or equal.

\begin{proof}[Proof of Proposition~\ref{prop:markov_kernel}]
We are going to prove that 
there exists at most one probability measure $\pi$ on $\CK_L$ that is invariant under $\Psi$ and such that $\pi$ a.e.\ $K_1$ is good.  Such a measure $\pi$ being necessarily a mixture of extremal invariant ergodic measures supported on good sets, it will suffice to see that there exists at most one extremal ergodic measure supported on good sets.  

Let $\nu$, $\widetilde \nu$ be two extremal ergodic measures supported on good sets.  Let us choose $K_1$ and $\widetilde{K}_1$ independently according to these two probability measures (on the same 
probability space) and then let us first apply (independently) the first step $\Phi$ of $\Psi$ to construct $K_2$ and $\widetilde K_2$.
  
Note that for every good $K_1$ there exists  $\delta (K_1) > 0$ such that with probability at least $\delta$, $K_2$ is a subset of the right-hand half $D_+ = (1/2, 1 ] \times [0, \ell]$ of the rectangle $D$ (recall that finitely many $O_j$'s have diameter at least $1/2$). Hence, with a random but positive conditional probability (given $K_1, \widetilde K_1$), both $K_2$ and $\widetilde K_2$ are subsets of $D_+$.  But, then, conditionally on this event, it is possible (simply using absolute continuity of $\SLE_\kappa (\rho_1 ; \rho_2)$ processes defined in two different domains) to see that one can couple the second iteration step that constructs $K_1'$ and $\widetilde K_1'$ in such a way that these two sets do coincide (and stay in the left-hand half of $D$) with positive probability.  But since the laws of $K_1'$ and $\widetilde K_1'$  are respectively equal to $\nu$ and $\widetilde \nu$, this shows that these two measures are not singular, which implies (because extremal ergodic measures are either 
singular or equal) that they are equal.
\end{proof}

\subsection{Second variant}
\label{subsec:bc_three_paths}

We now explain a closely related result, which will be derived using a variation of the argument used to prove Theorem~\ref{thm:bc_two_paths}.  Throughout, we suppose that $\kappa \in (8/3,4)$ and $\kappa'=16/\kappa \in (4,6)$.  As we mentioned earlier, this version is relevant for the proof of Lemma~\ref{lem:sle_k_k_minus_6_reversal} in the case $\kappa' \in (4,6)$.

Let us consider the same rectangle $D$. We denote by $T$ and $B$ its top and bottom sides.  For $l \in (0, \ell)$, we also denote by $I_l$ the horizontal segment $[0,1] \times \{ l \}$. 

When $\eta'$ is a continuous non-self-crossing and non-self-tracing path in $\overline D$ from $0$ to $i \ell$, we say that a connected component of $D \setminus \eta'$ is {\em to the right of $\eta'$} if its boundary contains an open interval of $(T \cup R \cup B ) \setminus \eta'$. 

Suppose that we have a law on pairs $(\eta',\Gamma)$ where $\eta'$ is a continuous non-self-crossing path in $\overline D$ from $0$ to $i \ell$ and 
$\Gamma$ is a collection of loops in the components of $D \setminus \eta'$ (i.e., each loop is in the closure of one of these components) which are to the right of $\eta'$ which satisfy the following properties:
\begin{itemize}
\item Given $\eta'$, the conditional law of $\Gamma$ is given independently by that of a $\BCLE_\kappa(-\kappa/2)$ in each of the components of $D \setminus \eta'$ which are to the right of $\eta'$ with marked points given by the endpoints of the interval of the component boundary which is contained in $\partial D$.
\item If we condition on the loops of $\Gamma$ which intersect $T$, then the conditional law of $\eta'$ in the remaining domain is that of an $\SLE_{\kappa'}(\kappa'-6)$ from $0$ to $i \ell$ with a single force point at the right-most point $y$ of $B$ so that no loop of $\Gamma$ intersects both $T$ and $[0, y)$.  
\item If we condition on the loops of $\Gamma$ which intersect $B$, then the conditional law of the time-reversal of $\eta'$ in the remaining domain is that of an $\SLE_{\kappa'}(\kappa'-6)$ from $i \ell$ to $0$ with a single force point at 
at the right-most point $y$ of $T$ so that no loop of $\Gamma$ intersects both $B$ and $[i \ell, y)$.  
\end{itemize}

Then, we see that the law of $\eta'$ is invariant under two different operations: 
\begin{itemize}
\item Sample $\Gamma$ given $\eta'$, keep only the loops that touch $T$, and then resample $\eta'$.
\item Sample $\Gamma$ given $\eta'$, keep only the loops that touch $B$, and then resample $\eta'$. 
\end{itemize}

Exactly as in the previous argument, one can see that these two resampling operations correspond to two Markovian kernels that we denote by $\Psi_1$ and $\Psi_2$ (one can view the paths in question as corresponding to compact sets, define good sets in a similar manner as before, and define the operation in a measurable way, and see that it coincides with the above description in the case where the set is a continuous path). Then: 

\begin {proposition}
There exists at most one probability measure $\pi$ that is invariant under both $\Psi_1$ and $\Psi_2$ and that is supported on the set of continuous non-self-crossing curves from 
$0$ to $i \ell$ in $\overline D$. 
\end {proposition}

Most of the proof of this statement is almost identical to the previous arguments, except for the final resampling argument, that we now describe in more detail:

Suppose that $\eta'$ and $\wt \eta'$ are two independent samples of two probability measures $\pi$ and $\wt \pi$ that satisfy the conditions of the proposition. In order to prove that $\pi = \wt \pi$, it is sufficient to show that these two probability measure do not have disjoint support.  For this we just need to show that by performing a resampling step corresponding to  $\Psi_1$ and then a resampling step corresponding to $\Psi_2$, we can arrange so that the obtained paths coincide with positive probability.  First, we do a resampling step corresponding to $\Psi_1$ as follows. 
\begin{itemize}
\item Given $\eta$ and $\wt \eta$, we sample  $\Gamma,\wt{\Gamma}$ independently. 
Then there is a positive chance that no loops of $\Gamma$ and $\wt \Gamma$ intersects both $T$ and $I_{3\ell/4}$. We call $E'$ this event. 
\item We then resample $\eta'$ and $\wt \eta'$ using the second part of the resampling step $\Psi_1$. We do this independently except when the event $E'$ occurred. 
Let us denote $\tau$, $\wt \tau$ the respective hitting times of $I_{\ell/2}$ by $\eta'$ and $\wt \eta'$. 
By standard absolute continuity properties for $\SLE$, when $E'$ holds,  we can couple $\eta'|_{[0,\tau]}$ with $\wt{\eta}'|_{[0,\wt{\tau}]}$ so that with positive probability, 
 these two portions agree and do intersect $R$ (so that they disconnect $T$ from $B$) -- we call $E$ this event. So, we do couple them in this way. 
\end {itemize} 
We note that for any two continuous curves $\eta'$ and $\wt \eta'$, the probability that in this resampling step corresponding to $\Psi_1$, we obtain that the probability of the event $E$ is strictly positive. 
We then perform a second resampling step corresponding to $\Psi_2$: On the event $E$, we can couple the loops of $\Gamma$, $\wt{\Gamma}$ that intersect $B$ in such a way that they are identical (this is just because they are defined in identical domains). 
We can then resample $\eta'$ and $\wt \eta'$ in the second step of $\Psi_2$ in such a way that they coincide. 

To conclude, we can for instance define the resampling kernel $\Psi$ as follows. With probability $1/3$, we do not do anything and keep the configuration as it is, with probability $1/3$ we apply $\Psi_1$ and with probability $1/3$ we apply $\Psi_2$, 
and apply the fact that two extremal ergodic invariant measures with respect to $\Psi$ are either mutually singular or equal. The measures $\pi$ and $\wt \pi$ are then on the  one hand invariant under $\Psi$, 
and the previous argument shows that they are not singular, so that they are necessarily equal.

\section{Some details for Section~\ref{subsec:paths_together}} 
\label{AppB}

We now indicate in a rather informal way the type of ideas that enable to derive the claims made in Section  \ref {subsec:paths_together} about the existence of the constants  $c_3$ and $c_1$.
The goal of this section is to outline the main steps of possible proofs rather than providing full lengthy details. 

\subsection{The existence of $c_1$} 
Let us first indicate one way to derive the claim about the existence of the constant $c_1$ related to the event $G_1$. Recall that the goal is to obtain a lower bound for the  $\p_{\ul z}$ probability of a certain event $G_1$, uniformly over all 4-tuples $\ul z$ of starting points in ${\mathcal T}_{\delta_0}$. 

Let us write $\delta_0' = \delta_0^{1/2}$ and $\delta_0''= \delta_0^{1/4}$. 
A first idea is to show that (provided $\delta_0$ was chosen small enough),  the $\p_{\ul z}$ probability of the event $X$ that the whole strands
stay in the $\delta_0''$ neighborhood of the quarter circles joining $-i$ and $1$, and $i$ and $-1$ respectively is bounded from below, uniformly with respect to  $ \ul z \in {\mathcal T}_{10 \delta_0}$. 
To see this, we can first note that if we consider the configuration with starting 4-tuple $\ul o$, then by the resampling arguments, one can see 
that the $\p_{\ul o}$ probability that the strands do stay in the $\delta_0'$-neighborhood of these quarter-circles is positive (see Figure \ref {sketchAppB0}).  Then, on the event where the entire strand originating from $i$ stays in the $\delta_0'$-neighborhood of its quarter-circle, the conditional law of the strand 
originating at $-i$ is an $\SLE_\kappa$ in the complement of the first strand. In particular, with a positive 
probability (bounded from below), if one explores (without having explored the strand originating from $i$)  the strand originating at $-i$ until it reaches distance $30 \delta_0$ from $-i$, one did on the way 
reach for any $\ul{z} \in {\mathcal T}_{10 \delta_0}$  a tip configuration
that is conformally equivalent to any $\ul z$, and then came back to a configuration that is conformally equivalent to $\ul o$ (note that one just needs to control the cross-ratio between the four 
tips). From this, one gets readily the fact that  $\p_{\ul z}[X]$ is bounded 
from below uniformly in $\ul z \in {\mathcal T}_{10 \delta_0}$.

\begin{figure}[ht!]
\begin{center}
\includegraphics[width=0.35\textwidth]{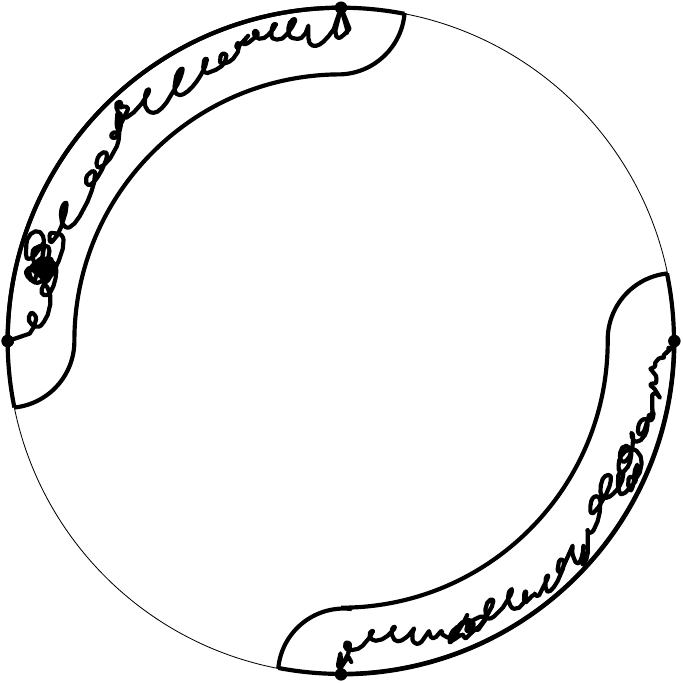} 
\quad
\includegraphics[width=0.35\textwidth]{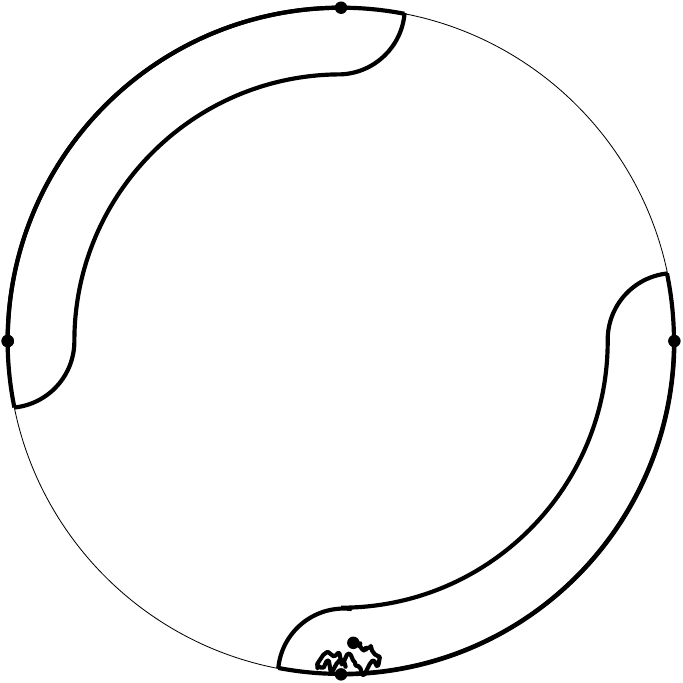} 
\end{center}
\caption{\label{sketchAppB0} The two strands defined under $\p_{\ul o}$ do stay in the neighborhood of the quarter-circles (left).  The exploration near $-i$ that enables to get a lower bound for $\p_{\ul z}[X]$
uniform with respect to 
$\ul z \in {\mathcal T}_{10 \delta_0}$ (right). }
\end{figure}
Let us now suppose that $\ul z \in {\mathcal T}_{\delta_0}$. Recall that $z_1$ is then the point that is $\delta_0$-close to $-i$ and that $z_3$ is $\delta_0$-close to $i$.
When $X$ holds, we can first explore the entire strand that starts from $z_3$, and then the conditional law of 
the one that starts from $z_1$ will be that of an $\SLE_\kappa$ in the remaining domain, and this remaining domain will contain the $\delta_0''$ neighborhood of the quarter-circle from $-i$ to $1$.
We can then use basic absolute continuity estimates between SLE curves in various domains in order to 
conclude that on this event, the law of the strand originating at $z_1$ up to the first time at which it 
reaches distance $99\delta_0$ from $z_1$ is absolutely continuous with respect to that of a radial $\SLE_\kappa$ in the unit disk (targeting the origin) until this hitting time, and the 
Radon-Nikodym derivative between these two laws will be bounded from below and above by absolute constants.

A next step is to notice that for a radial $\SLE_\kappa$ started from $z_1$ and stopped at the deterministic time $\delta_0^2$ (measured by the usual log-conformal radius from the origin as customary for radial Loewner chains), the new position $w$ of the tip (after mapping back via the uniformizing map $\phi_w$ normalized at the origin) is that of a Brownian motion on the unit circle at time $\kappa \delta_0^2$ that was started at $z_1$. Furthermore, on the event that at that time, the probability that the SLE did not yet exit the $99 \delta_0$ neighborhood of $z_1$ is positive, and the conditional law (given that event) of $w$ has a density that is bounded from below by an absolute constant on the arc of length $5 \delta_0$ centered around $z_1$. A further remark is that on this event, the map $\phi_w$ will not move nor distort the neighborhoods of the other three points $z_2$, $z_3$ or $z_4$ by much. More precisely, any point in the $10 \delta_0$ neighborhood of those points will not move by more than $\
delta_0$, and $|\phi_w'|$ will be close to $1$ there (the fact that it is uniformly greater than $1/2$ will be sufficient here). So, on this event, we can think of the new configuration of tips to have points that are at distance less then $\delta_0$ from $z_2$, $z_3$ and $z_4$, and one point which is distributed according to a law that has a density that is bounded from below on the $5 \delta_0$ neighborhood of $z_1$ (importantly, when we will then take the image of this point three times by a map with derivative bounded from below by $1/2$ and that moves the points by less than $\delta_0$, we will still have a density that it bounded from below by a constant on the $\delta_0$-neighborhood of $-i$).

After growing the slit from $z_1$ up to the time $\delta_0$ and then mapping back via $\phi_w$, we can do the same procedure for the strand starting from $\phi_w (z_2)$, and then after mapping back, apply the same procedure iteratively for the two remaining strands. 

Wrapping things up, we get readily that there exists an event $G_1'$ with a probability $p_1 (\ul z) := \p_{\ul z}[G_1']$ that is bounded from below by some constant $\wt p_1$ uniformly with respect to $\ul z \in {\mathcal T}_{\delta_0}$ such that at the end of this procedure, the conditional law of the four new tips $\ul z^1$  has a density $p_2 ( \ul z^1 )$ on ${\mathcal T}_{\delta_0}$ (with respect to the product measure) that is bounded from below by a constant $\wt p_2$. 

Finally, we use a uniform random variable $U$ on $[0,1]$ in order to define the event $G_1$: We define $G_1$ to be  the event $G_1'$ holds and that $U < \wt p_1 \wt p_2 / (p_1 ( \ul z) p_2 (\ul z^1))$ where $\ul z$ is the 4-tuple of starting points and $\ul z^1$ the obtained 4-tuple corresponding to the new tips. In this way, the conditional law of $\ul z^1$ given $G_1$ is indeed uniform on ${\mathcal T}_{\delta_0}$, and the probability of $G_1$ is equal to some positive constant that is independent of $\ul z \in {\mathcal T}_{\delta_0}$.

\subsection{The existence of $c_3$}
Let us now describe ideas that provide the argument for the existence of the constant $c_3$ related to the event $G_3$. 
When $r <1$, $D_r$ and $C_r$ will denote here the disk of radius $r$ around the origin and the circle of radius $r$ around the origin. 

We first note that for each given $\delta' \le \delta_1 / 100$ (which should be thought of as very small compared to $\delta_1$), the set of starting points that are $\delta_1$-separated can be covered by some finite family of sets of starting points $\ul z$ where each $z_j$ lies in some arc of length $\delta'$ and any two arcs are at distance greater than $9 \delta_1/10$ from each other. So, it suffices to prove the lower bound for each of these products-of- $\delta'$-arcs of starting points separately.

Let us fix four such $\delta'$-long boundary arcs. We can then find, for each choice of these four arcs, a pair of well-chosen disjoint deterministic ``tubes of width $\delta_1/4$'' as depicted in 
Figure~\ref{sketchAppB} such the following hold that:
\begin{itemize}
\item They join the boundary arcs pairwise,
\item They reach $C_{1/3}$ around the origin, but not $C_{1/6}$,
\item They stay at distance $\delta_1/4$ of each other,
\item The connected components of the intersections of the tubes in $D_1 \setminus D_{1/2}$ that have respectively $z_1$, 
\ldots, $z_4$ on their boundaries are disjoint and at distance at least $\delta_1/4$ from each other, and
\item They have the property that the distance between each of the four $\delta'$-boundary arcs and the complement of the tubes in the unit disk is at least $\delta_1/10$. 
\end{itemize}

Then, for each pair of such tubes, the probability of the event $U$ that the
two whole strands do stay in the union of these two tubes is positive and bounded from below,
independently of the 4-tuple points $\ul z$ chosen in the four given arcs of length $\delta'$ (using essentially the same ideas as in the case of the event $X'$ that we outlined above).  

Now suppose that we wish to check whether $U$ holds, but do not do it completely, using the following algorithm:  One first explores totally the strand starting from $z_3$, and if it stayed in its tube, then one starts exploring the other arc starting from $z_1$.  Suppose that this second arc went through $C_{1/3}$ and  that  we then stop at the moment when it reaches again $C_{7/12}$ while in the part of the tube near $z_2$ or $z_4$,
and that it stayed in its tube as well until that moment. Let us call $V$ the event that this happens. Note that 
$U \subset V$, so that the probability of $V$ is bounded from below by the probability of $U$.
\begin{figure}[ht!]
\begin{center}
\includegraphics[width=0.35\textwidth]{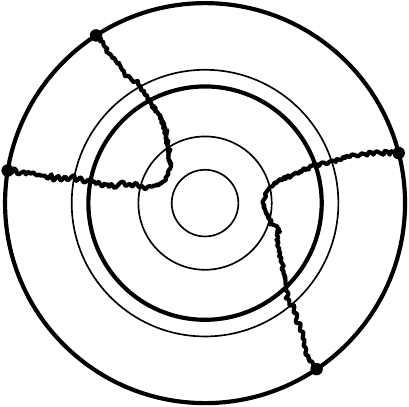} 
\quad
\includegraphics[width=0.35\textwidth]{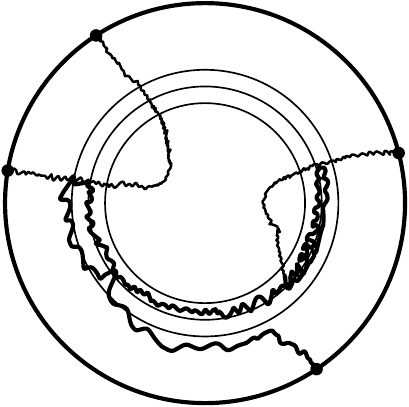} 
\vskip 8mm
\includegraphics[width=0.35\textwidth]{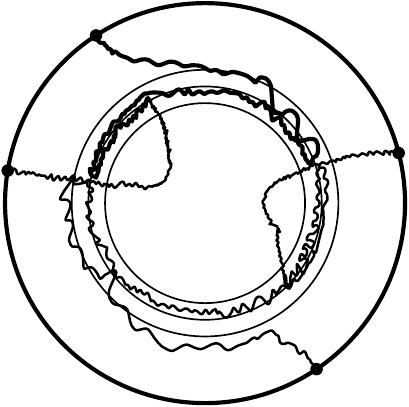} 
\quad
\includegraphics[width=0.35\textwidth]{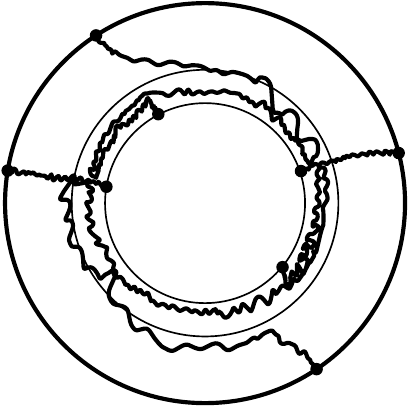} 
\end{center}
\caption{\label{sketchAppB} Sketch of the resampling argument to obtain $G_3$. The two arcs in a tube (and the circles of radii $1/6$, $1/3$, $7/12$ and $2/3$), and then the two steps of the resampling create the event $G_3$. }
\end{figure}Also, the conditional law (given $V$) of the remaining final part of the second strand is just an SLE. 
Using conformal estimates, one can find a lower bound on the conditional probability that this
SLE will not reenter the disc $D_{1/2}$ again, and that before hitting $C_{2/3}$, 
the union of this SLE with the first strand disconnects the whole boundary arc from $z_1$ to $z_3$ 
from the origin in the unit circle (because the probability that the SLE hits the 
corresponding boundary parts of the domain is positive, see Figure \ref {sketchAppB} for a sketch). Hence, the probability of the obtained event $W$ is 
bounded from below.

Then, we can suppose that the event $W$ holds, and we can decide to discover the configuration
by first discovering the strand starting from~$z_1$ totally, and then to discover the strand starting from $z_3$, 
almost until the end and to resample its end part in a similar way as before (see again Figure \ref {sketchAppB}).
Then, the obtained configuration is indeed in~$G_3$. 
Wrapping this up, we get that the probability of~$G_3$ is bounded from below by some absolute constant $c_3$.

\section*{Acknowledgements}
JM and WW respectively thank the FIM at ETH Z\"urich and the Statslab of the University of Cambridge for their hospitality on several occasions during which part of the work for this project was completed.   JM also thanks Institut Henri Poincar\'e for support as a holder of the Poincar\'e chair, during which this work was completed.  JM's, SS's and WW's work were respectively partially supported by the NSF grant DMS-1204894, the NSF grant DMS-1209044 and the SNF grants 155922 and 175505. WW\ is part of the NCCR Swissmap. 
We also express many thanks to the referees for their work and very useful comments.

\bibliographystyle{abbrv}

\end{document}